\documentclass[12 pt,reqno]{amsart}
\usepackage{amssymb,amsmath,amstext,amsgen,amsfonts,amsthm,fullpage,enumerate,mathrsfs,verbatim,url}
\usepackage{ctable}
\usepackage[nobysame,alphabetic]{amsrefs}

\newcommand{\N}{\mathbb{N}}
\newcommand{\bR}{\mathbb{R}}
\newcommand{\R}{\mathbb{R}}

\newcommand{\Z}{\mathbb{Z}}
\newcommand{\bZ}{\mathbb{Z}}
\newcommand{\T}{\mathbb{T}}
\newcommand{\bT}{\mathbb{bT}}
\renewcommand{\H}{\mathbb{H}}
\renewcommand{\H}{\mathbb{H}}

\newcommand{\epsilonSubS}{{\epsilon_0}}

\newcommand{\cF}{\mathcal{F}}
\newcommand{\allballs}{\mathcal{\hat{G}}}
\newcommand{\smallballs}{\mathcal{G}}
\newcommand{\flatballs}{{\mathcal{G}_2}}
\newcommand{\curvyballs}{{\mathcal{G}_1}}

\newcommand{\cB}{\mathcal{B}}
\newcommand{\cH}{\mathcal{H}}

\newcommand{\net}{\mathbb{X}}
\newcommand{\ball}{B}
\newcommand{\rad}{r}
\newcommand{\dist}{d}
\newcommand{\jump}{J}
\newcommand{\child}{\mathcal{D}}

\renewcommand{\epsilon}{\varepsilon}
\renewcommand{\phi}{\varphi}
\renewcommand{\emptyset}{\varnothing}

\newtheorem{theorem}{Theorem}[section]
\newtheorem{proposition}[theorem]{Proposition}
\newtheorem{lemma}[theorem]{Lemma}

\newtheorem{cor}[theorem]{Corollary}
\newtheorem{maintheorem}{Theorem}

\newtheorem{remark}[theorem]{Remark}
\newtheorem{note}[theorem]{Note}

\DeclareMathOperator{\diam}{diam}
\DeclareMathOperator{\cent}{Center}

\begin{document}

\title{THE TRAVELING SALESMAN PROBLEM IN THE HEISENBERG GROUP: UPPER BOUNDING CURVATURE}
\author{Sean Li \and Raanan Schul}
\date{\today}
\address{Department of Mathematics, The University of Chicago, Chicago, IL 60637}
\email{seanli@math.uchicago.edu}
\address{Department of Mathematics, Stony Brook University, Stony Brook, NY 11794-3651}
\email{schul@math.sunysb.edu}

\begin{abstract} 
We show that if a subset $K$ in the Heisenberg group (endowed with the Carnot-Carath\'{e}odory metric) is contained in a rectifiable curve, then it satisfies a modified analogue of Peter Jones's geometric lemma.  
This is a quantitative version of  the statement that a finite length curve has a tangent at almost every point.
This condition complements that of \cite{FFP} except a power 2 is changed to a power 4.  Two key tools that we use in the proof are a geometric martingale argument like that of \cite{Schul-TSP} as well as a new curvature inequality in the Heisenberg group.
\end{abstract}

\maketitle


\section{Introduction}

Let $\H$ denote the Heisenberg group, endowed with the Carnot-Carath\'{e}odory distance.  For $B=B(x,r) \subset \H$, a (closed) ball  of radius $r$ centered at $x$, and a set $K$ we define $\beta_K(B)$ to be
\begin{align*}
 \beta_K(B)=\inf\limits_{L} \sup\limits_{x\in K\cap B} \frac{\dist(x,L)}{\diam(B)},
\end{align*}
where the infimum is taken over all {\it horizontal lines} $L$.   We will describe both the horizontal lines and the metric for the Heisenberg group in the next section.
The number  $\beta_K$ is a coarse notion of curvature associated to the ball $B$. 
This notion of curvature is tested on a fixed scale $r$, the radius of $B$.
A natural thing to consider is  looking at balls of many scales. The topic of this paper is an upper bound for an integral or sum of this notion of curvature, where the sum is  over all scales and locations.  This is not a new idea, and we discuss its long history later in this section.

A set $\Gamma$ is called a {\it rectifiable curve} if it is the  image of a finite length curve, or, equivalently, the Lipschitz image of a finite interval.
We will use $\cH^k$ to denote the $k$-dimensional Hausdorff measure and $\ell(\gamma)$ to denote the arclength of a curve $\gamma$.
In this paper we prove the following theorem.
\begin{maintheorem} \label{t:maintheorem-INT}
There is a constant $C>0$ such that for any  rectifiable curve $\Gamma$
the following holds.
We have 
\begin{align}
\int\limits_{\H} \int\limits_0^{+\infty} \beta_\Gamma(B(x,t))^4\frac{dt}{t^4}d\cH^4(x) 
\leq C\cH^1(\Gamma).  \label{e:sum-beta-upper-bound-INT}
\end{align}
\end{maintheorem}
We remark that the use of Hausdorff measure of dimension 4 directly corresponds to the Hausdorff dimension of $\H$ and the  power of $t$.  However, it does not correspond to the power 4 of $\beta$.  That 4 comes from the modulus of curvature coming directly from the Heisenberg geometry.  In an $n$-dimensional Euclidean space,   the same formula holds where the power of $\beta$ is 2, and the power of $t$ as well as the Hausdorff measure dimension are $n$ \cite{Jones-TSP,  Ok-TSP} (see Section \ref{s:background}).



\subsection{Background}\label{s:background}
We give below a brief survey for a subject which can easily  (and does) fill books (see e.g. \cite{Pajot} for a nice exposition). 

A result of this type was first proven by Jones in \cite{Jones-TSP}.   There he showed 
that  a bounded set $K\subset\bR^2$ is a subset of a rectifiable curve in $\bR^2$ if and only if
\begin{equation*} 
\int\limits_{\bR^2} \int\limits_0^{+\infty} \beta_K(B(x,t))^2 \frac{dt}{t^2}d\cH^2(x) <\infty\,.
%
\end{equation*} 
Moreover, there exists a constant $C>0$ (independent of $K$) such that
\begin{itemize}
\item 
for every connected set $\Gamma$ containing $K$
$$\int\limits_{\bR^2} \int\limits_0^{+\infty} \beta_K(B(x,t))^2 \frac{dt}{t^2}d\cH^2(x) 
\leq C \cH^1(\Gamma)$$ 
\item  there exists (with explicit construction) a connected set $\Gamma\supset K$ such that
$$\cH^1(\Gamma)\leq C\left(\diam(K) +\int\limits_{\bR^2} \int\limits_0^{+\infty} \beta_K(B(x,t))^2 \frac{dt}{t^2}d\cH^2(x)\right)$$
\end{itemize} 
Above, $\beta_K$ is defined as in $\H$, except we take infimum over all lines in $\R^2$. 
The result actually holds in $\bR^n$: the existence of such a $\Gamma$ above actually worked for $\bR^n$, and Okikiolu \cite{Ok-TSP} completed the upper bound on the integral to hold in $\bR^n$ as well.  In $\bR^n$, the integral  in question is 
$$\int\limits_{\bR^n} \int\limits_0^{+\infty} \beta_K(B(x,t))^2 \frac{dt}{t^n}d\cH^n(x)\,.$$ 
We note that these results were actually phrased using sums and not integrals:
There exists a constant $C>0$ (independent of $K$) such that
\begin{itemize}
\item 
for every connected set $\Gamma$ containing $K$
$$\sum_{Q\in\Delta(\bR^n)}\beta_K(3Q)^{2}\diam Q 
\leq C \cH^1(\Gamma)$$ 
\item  there exists (with explicit construction) a connected set $\Gamma\supset K$ such that
$$\cH^1(\Gamma)\leq C\left(\diam(K) +\sum_{Q\in\Delta(\bR^n)}\beta_K(3Q)^{2}\diam Q\right).$$
\end{itemize}
where $\Delta(\bR^n)$ is the collection of dyadic squares. This two sided inequality is known as the geometric/analytic traveling salesman theorem (in $\R^n$) and has had applications in the study of harmonic measure \cite{BJ}.
This result was generalized to a characterization of subsets of a rectifiable curves in Hilbert space by  the second author \cite{Schul-TSP} (where one must replace cubes with a family of balls centered on the set, as in eq. \eqref{e:sum-beta-upper-bound} below).  A rich theory connecting the above with singular integrals was developed by David-Semmes and others \cite{DS91, DS93}.  In the last decade people have sought to generalize this to general metric spaces \cite{Ha-AR, Schul-AR,Ha-non-AR}.  

One particular metric space where this phenomenon has been studied is the Heisenberg group.  In that setting, there are many analogues of Euclidean notions, including translation, dilation structure, and horizontal lines.  Thus, it can be hoped that the Euclidean results would translate over.  In \cite{FFP}, the authors showed one side of the traveling salesman inequality, that is, $K$ is contained in a rectifiable curve in $\H$ if
\begin{align}
\int\limits_{\H} \int\limits_0^{+\infty} \beta_K(B(x,t))^2 \frac{dt}{t^4}d\cH^4(x) <\infty\,.
\end{align}
Moreover, there exists a connected set $\Gamma \supseteq K$ such that
\begin{align}
\cH^1(\Gamma) \leq C\left(\diam(K) + 
\int\limits_{\H} \int\limits_0^{+\infty} \beta_K(B(x,t))^2 \frac{dt}{t^4}d\cH^4(x) 
\right) \label{FFP-ineq}
\end{align}
for some universal constant $C > 0$. 
It was probably natural then to expect that the opposite inequality in the traveling salesman theorem also holds in the Heisenberg group.  However, Juillet constructed in \cite{juillet} a sequence of rectifiable curves $\{\Gamma_n\}_{n=1}^\infty$ such that $\cH^1(\Gamma_n)$ was bounded but
\begin{align}\label{juillet-ineq}
\int\limits_{\H} \int\limits_0^{+\infty} \beta_{\Gamma_n}(B(x,t))^2 \frac{dt}{t^4}d\cH^4(x) \to \infty\,.
\end{align}
(The results in \cite{juillet} were actually given with an equivalent sum replacing the integral).

In this context, our result shows that the missing inequality is true if we change the power of the exponent from 2 to 4.  The motivation for such a modification follows from a similar result of \cite{Li}, where one had access to the function of the curve itself rather than just the image of the curve.  There, it was shown that a parametric version of the main theorem holds where the $\beta$ numbers measured the deviation of a subcurve from a horizontal ``affine'' function with respect to its domain.  In the same paper, the power of $\beta$ was related to the Markov convexity of the target space.  For the Heisenberg group, the Markov convexity was recently calculated to be 4 \cite{Li-2}.  In this paper, we will not use convexity but rather the related notion of curvature.  As a very rough description, the proof in this paper follows that of both  \cite{Schul-TSP} and \cite{Schul-AR}, however the technicalities involved are different.

A classical understatement is that Jones's traveling salesman theorem \cite{Jones-TSP} is {\it just the Pythagorean theorem}.  The Pythagorean theorem is used to estimate the excess in the triangle inequality by the height squared divided by the diameter (see Remark 1.2 in \cite{Schul-AR}).  As the Pythagorean theorem no longer holds in the Heisenberg group, one needs to derive a new curvature inequality, which is done in 
Part B.
Further technicalities arise as two horizontal line segments in the Heisenberg group whose endpoints are $\epsilon$ apart may be as far apart as $\sqrt{\epsilon}$ in the middle. 

\begin{remark}
It should be noted that the power 4 in \eqref{e:sum-beta-upper-bound-INT} cannot be improved.  Indeed, with a minor modification, the construction in \cite{juillet} (taking the parameter in the construction to be $\theta_k = \frac{c}{k^q}$ for $q > 1/2$ instead of $q = 1$) yields a sequence of rectifiable curves with bounded $\cH^1$ measure such that \eqref{juillet-ineq} holds for any modified power $p < 4$.
\end{remark}

Following our work in this paper, the central question of the traveling salesman in the Heisenberg group now becomes whether \cite{FFP} is tight.  As \eqref{e:sum-beta-upper-bound-INT} is known to be essentially tight, it seems like it may be possible to improve the power of \eqref{FFP-ineq} to 4, which would complete the traveling salesman theorem.  On the other hand, if a counterexample were to show that this hoped-for power-4 inequality is not true, then one can ask whether there may be another functional besides a weighted sum of powers of $\beta_K(B)$ that would acheive a two-sided theorem.  We hope to return to this question in a future paper  \footnote{See \cite{Li-Schul-construction-bounds} for an improvement of \eqref{FFP-ineq} to any power of $\beta$ which is less than $4$}.

\subsection{About constants}
There are many constants that are introduced throughout the paper.  These are set and discussed in a special environment which is denoted by {Note X.Y}.  There are 7 such notes throughout the paper.

The paper has two parts: Part A and Part B.  
Part B  has the sole purpose of proving  Proposition \ref{p:prop-4}.  It is independent of Part A (other than a general introduction of notation and known statements).  This proposition is separated out to ensure that the order of determining constants in clear.

\subsection{Acknowledgements}
S. Li was supported by a dissertation fellowship from New York University's Graduate School of Arts and Sciences and a postdoctoral research fellowship NSF DMS-1303910.
R.~ Schul was partially supported by a fellowship from the  Alfred P. Sloan Foundation as well as  by NSF  DMS 11-00008.
Some of this work was completed while the second author was visiting IPAM.
The authors wish to thank the referee for the detailed report as well as suggesting a better proof for Lemma \ref{l:lemma-8}. The referee's efforts helped fixed numerous mistakes as well as improve the exposition.

\section*{PART A}

\section{Preliminaries}

\subsection{The Heisenberg group}

The Heisenberg group is the three dimensional Lie group $(\bR^3,\cdot)$ where the group product is
\begin{align*}
  (x,y,z) \cdot (x',y',z') = \left(x+x',y+y',z+z' + \frac{1}{2}(xy'-x'y) \right).
\end{align*}
It can be immediately verified that the origin is also the identity 0.

There exists a natural path metric on the Heisenberg group that we define as such.  Using the smoothness of left multiplication, we can define a left-invariant subbundle $\Delta$ of the tangent bundle by letting $\Delta_0$ be the $xy$-plane.  Further endow $\Delta$ with a left-invariant scalar product $\{\langle \cdot, \cdot \rangle_x\}_{x \in \H}$.  Then given two points $x,y \in \H$, we can define the Carnot-Carath\'{e}odory distance between them as
\begin{align*}
  d_{cc}(x,y) := \inf \left\{ \int_a^b \langle \gamma'(x), \gamma'(x) \rangle_{\gamma(x)} dx : \gamma\in C^1([a,b];\H), \gamma(a) = x, \gamma(b) = y, \gamma'(x) \in \Delta_{\gamma(x)} \right\}.
\end{align*}
All continuous paths $\gamma : I \to \H$ that satisfy $\gamma'(x) \in \Delta_{\gamma(x)}$ (almost everywhere) are called {\it horizontal paths}.  One natural point of worry is whether there always exists such paths connecting any two points in the Heisenberg group.  Chow's theorem states that $d_{cc}(x,y) < \infty$ for all $x,y \in \H$ (see {\it e.g.} \cite{montgomery}).  As we are taking the Riemannian length over a subclass of curves, this geometry is sometimes called sub-Riemannian geometry.

We will not work directly with the Carnot-Carath\'{e}odory metric.  Instead, we define another metric that is biLipschitz equivalent to it.  The advantage of this new metric is that it is easy to calculate explicit distances between points.  Note that proving Theorem \ref{t:maintheorem-INT} for the new metric will also prove it for the Carnot-Carath\'{e}odory metric as the notion of rectifiability is preserved when passing to a biLipschitz equivalent metric.  All the biLipschitz constants will be absorbed into the constant $C$.
Note \ref{n:n2.4} is the place where we make the final choice of the metric we will be using.

For a given $\eta > 0$, we define
\begin{align*}
  N : \H &\to \R_+ \\
  (x,y,z) &\mapsto \left((x^2+y^2)^2 + \eta z^2\right)^{1/4}
\end{align*}
to be the Koranyi norm.  The following proposition is the result of \cite{cygan} and contains the definition of our metric.
\begin{proposition} \label{koranyi}
  The Koranyi metric $d(g,h) = N(g^{-1}h)$ is a left-invariant semimetric that is bi-Lipschitzly equivalent to the Carnot-Carath\'{e}odory metric.  If, in addition, $\eta \in (0,16]$, then the triangle inequality holds.
\end{proposition}

We will require $\eta$ to be sufficiently small.
We will fix $\eta$ in Note \ref{n:n2.4}.

For every $\lambda > 0$, we have the automorphism
\begin{align*}
  \delta_\lambda : \H &\to \H \\
  (x,y,z) &\mapsto (\lambda x,\lambda y,\lambda^2 z).
\end{align*}
Note that $\delta_\lambda$ scales the Koranyi metric, {\it i.e.}
\begin{align*}
  d\left( \delta_\lambda(g),\delta_\lambda(h) \right) = \lambda d(g,h).
\end{align*}

An important feature of the Heisenberg group is that there is a family of lines, called the horizontal lines, that are isometric to $\R$.  Before we define a horizontal line, we first define the horizontal elements of $\H$ to be those that are in the $xy$-plane.  For horizontal elements $(x,y,0) \in \H$, we can extend the parameter range of $\delta_\lambda$ to get $\delta_\lambda : (x,y,0) \mapsto (\lambda x,\lambda y,0)$ for any $\lambda \in \R$.  Then the horizontal lines of $\H$ are simply sets of the form $L = \{g \cdot \delta_t(h) : t \in \R\}$ when $g,h \in \H$ and $h$ is horizontal.  We can similarly define horizontal line segments.

There exists a projection homomorphism
\begin{align*}
  \pi : \H &\to \R^2 \\
  (x,y,z) &\mapsto (x,y).
\end{align*}
One can easily verify using the definition of the Koranyi norm that $\pi$ is 1-Lipschitz and even isometric when restricted to any horizontal line of $\H$.

We will also define the following map, which maps an element to the horizontal element ``below'' it:
\begin{align*}
  \tilde{\pi} : \H &\to \H \\
  (x,y,z) &\mapsto (x,y,0).
\end{align*}
Note that this is {\it not} a homomorphism.  We easily get that $N(\tilde{\pi}(g)) \leq N(g)$.

We need to define the following notion of horizontal interpolation.  Given $a,b \in \H$, we let
\begin{align*}
  \overline{ab} := \{a \delta_t \tilde{\pi}(a^{-1}b) : t \in [0,1]\}.
\end{align*}
Thus, for a subarc $\tau$ we have that $L_\tau = \overline{a(\tau)b(\tau)}$.

\begin{remark}\label{r:doesnt-hit}
  We stress that $\overline{ab}$ may not necessarily contain $b$ although it always contains $a$.  Indeed, $\overline{ab}$ is a horizontal line segment that starts from $a$ and goes in the horizontal direction of $a^{-1}b$.  Thus, $\overline{ab} \neq \overline{ba}$ unless $a$ and $b$ are co-horizontal.
\end{remark}

The following proposition, Proposition \ref{p:prop-4}, can be thought of as an improvement over the triangle inequality of the Koranyi metric.  The proposition is proven in Part B of the paper.  Other than the definitions above, the proof of Proposition \ref{p:prop-4} is independent of Part A.
This proposition is a curvature inequality in the Heisenberg group and should be thought of as a lower bound on the excess of the triangle inequality.
\begin{proposition}\label{p:prop-4}
  Let $\epsilon \in (0,1/2)$.  If $\eta \in \left(0, \left(\frac{\epsilon}{10}\right)^{10} \right)$, then for every $p_1,p_2,p_3,p_4 \in \H$ such that
  \begin{align}
    d(p_2,\{p_1,p_4\}) \geq \epsilon d(p_1,p_4), \notag \\
    d(p_3,\{p_1,p_4\}) \geq \epsilon d(p_1,p_4), \label{eps-lower-bound}
  \end{align}
  we have
  \begin{multline}
    d(p_1,p_2) + d(p_2,p_3) + d(p_3,p_4) - d(p_1,p_4) \geq \frac{\epsilon^4 \eta^2}{10^{14}\diam(\{p_1,p_2,p_3,p_4\})^3} \times \\
    \left(\max_{i \in \{1,2,3\}} \sup_{a \in \overline{p_ip_{i+1}}} d(a, \overline{p_1p_4})^4 \right). \label{curvature-ineq}
  \end{multline}
\end{proposition}
The condition \eqref{eps-lower-bound} says that the middle two points $p_2,p_3$ are not too close to the endpoints $p_1,p_4$.
We  remind the reader that the Koranyi metric $d$ in Proposition \ref{p:prop-4} above depends on the quantity $\eta$.  
\begin{note}\label{n:n2.4}
 From here on, we will fix $\eta=2^{-1200} < 2^{-1160}/10^{10}$, so that we may use the above proposition with $\epsilon=2^{-116}$, as is needed in its (only) application, which is to prove Lemma \ref{l:modified-prop-4}.
We remark that this choice of $\epsilon$ will be made in Lemma  \ref{l:modified-prop-4} when we set $\epsilon=\delta 2^{-J-6}$, where $\jump=100$ and $\delta = 2^{-10}$ are the parameters which  appear in Lemma \ref{l:build-filtration-arcs}.  We refer the reader to Note \ref{n:j-delta-const} for explanation of the setting of the $\jump$ and $\delta$ constants.  Note that having this fixed choice of $\eta$ also means of course that we have a fixed Koranyi metric $d$, which is the metric that will be used for the rest of Part A. In particular when we say `distance', `arc length' etc., these will be measured with respect to this Koranyi metric unless otherwise stated.  We remind the reader that proving Theorem I for this metric also proves it for the Carnot-Carath\'eodory metric, up to a constant that can be bounded by the biLipschitz distortion of the two metrics.
\end{note}


Another important feature of the Heisenberg group is that it is geometrically doubling.  Recall that a metric space $(X,d_X)$ is said to be geometrically doubling if there exists a constant $M \geq 1$ so that for every metric ball $B(x,r)$ can be covered by a set of no more than $M$ balls of half the radius
\begin{align*}
  B(x,r) \subseteq \bigcup_{i=1}^M B(y_i,r/2).
\end{align*}
Indeed, the Lebesgue measure of $\R^3$ is a Haar measure of $\H$.  This follows as group translations in $\H$ are affine transformations of $\R^3$ with determinant 1.  One can then see by looking at the anisotropic scaling of $\delta_\lambda$ that the volume of balls grows like $|B(0,r)| = cr^4$, which have polynomial growth.  A standard argument then shows that $\H$ must also be geometrically doubling.  It is well known that the Hausdorff dimension of $\H$ is 4 and that the 4-Hausdorff measure $\cH^4$ is also a Haar measure of $\H$.  Thus, $\cH^4$ is a constant multiple of the Lebesgue measure.

\subsection{Reduction to a special multiresolution}\label{s:special-multires}
We say that a set $\net$ is an {\it $\epsilon$ separated set} if whenever $x,y\in \net$ we have $\dist(x,y)\geq \epsilon$.
We say that a set $\net\subset K$ is an $\epsilon$ separated net {\it for $K$} if $\net$ is an $\epsilon$ separated set and for any $z\in K$ we have $x\in \net$ such that $\dist(x,z)<\epsilon$. 
For a given set $K$ and constant $A\geq 2$ we define a {\it multiresolution} $\allballs$  for $K$ as follows.  Let $\net_n$ be a $2^{-n}$ separated net for $K$ and assume that
$\net_{n+1}\supset\net_n$.  We then let
\begin{align*}
\allballs:=\{\ball(x,A2^{-n}):x\in\net_n,\ \ n \in \bZ\}\,.
\end{align*}
When it is important for us to emphasize $K$ we will write $\allballs^K$.  We will always omit $A$ from the notation, but remark that we will consider $A>2$ a fixed number (see Note \ref{n:fix-a}). We will refer to  $A$ as the {\it implied constant} of the multiresolution.
\begin{remark}
If the diameter of $K$ is, say, 1, we may construct a multiresolution for $K$ by choosing a single point for $\net_i$ where $i\leq 0$, and for $i>0$, choosing $\net_i$   inductively by taking  a max separated net. 
\end{remark}

We will show 
Theorem \ref{t:maintheorem-INT}  via the following lemma.

%
%
%
%
%
\begin{lemma}\label{l:equiv-thms}
Let $A>2$ be given.  
Let $K\subset\Gamma$ and let   $\allballs$ be defined using $K$ and implied constant $A$. 
If there is a constant $C_1<\infty$ such that 
\eqref{e:integral} holds, then  \eqref{e:allbals} holds, where $C_2<\infty$ depends only on $A$ and  $C_1$.
Conversly,
If there is a constant $C_2<\infty$ such that \eqref{e:allbals} holds for $K=\Gamma$, then
\eqref{e:integral} holds for a constant $C_1<\infty$ which depends only on $C_2$.

\begin{align}\label{e:integral}
\int\limits_{\H} \int\limits_0^{+\infty} \beta_\Gamma(B(x,t))^4\frac{dt}{t^4}d\cH^4(x) 
\leq C_1\cH^1(\Gamma).  
\end{align}

\begin{align}\label{e:allbals}
\sum\limits_{B\in \allballs} \beta_\Gamma(B)^4\diam(B)\leq C_2\cH^1(\Gamma). 
\end{align}

\end{lemma}
\begin{remark}
Eq. \eqref{e:allbals} will be shown with $C_2$ depending on $A$.  As an application of the above lemma  we will get \eqref{e:integral}  with $C_1$ depending on $A$.
\end{remark}

\begin{proof}
Let $x\in\H$ and $t\in[2^{-n-1},2^{-n})$, where $n\in \bZ$. 
If $K\cap \ball(x,t)\neq\emptyset$, then there is $z\in \net_n$ with $\dist(z,x)\leq 2^{-n}+t\leq 2^{1-n}<4t$.
Thus $\ball(x,t)\subset \ball(z,2^{2-n})\subset \ball(x,16t)$.
The lemma now reduces to a discretization of the double integral; this follows from a standard argument and the fact that 
for $\alpha\geq1$, we have $\beta_K(B) \leq \alpha \beta_K(\delta_\alpha(B))$,
as well as that the measure $\cH^4(\ball(\cdot,r))$ grows like a fixed constant times $r^4$.

\end{proof}

\begin{note}\label{n:fix-a}
For concreteness, we now fix $A=10$. Any constant $>2$ would suffice.
\end{note}

The remainder of this paper will be devoted to showing that for any $K\subset \Gamma$, which gives rise to $\allballs$, we have
\begin{align}
\sum\limits_{B\in \allballs} \beta_\Gamma(B)^4\diam(B)\leq C\cH^1(\Gamma).  \label{e:sum-beta-upper-bound}
\end{align}
for $C$ depending only on $A=10$ (and not on $K$ or the choice of $\allballs$).

\subsection{Metric space preliminaries}

Our definition of  $\beta_\Gamma(B)$ is scale independent in the sense that $\beta_{\delta_\alpha(\Gamma)}(\delta_\alpha(B)) = \beta_\Gamma(B)$.  As a corollary we get that we may suppose
without loss of generality that $\diam(\Gamma) = 1$ and the following lemma.
Let 
$$\smallballs= \left\{B\in\allballs:\rad(B)<\frac1{100}\right\}.$$
\begin{lemma}
 There exists some constant $C > 0$ depending only on the ambient metric space so that
 \begin{align*}
\sum\limits_{B\in \allballs\setminus\smallballs} \beta_\Gamma(B)^4\diam(B)\leq C\cH^1(\Gamma). 
\end{align*}
\end{lemma}
For a proof of this see the proof of Lemma 3.9 in \cite{Schul-TSP}, where this is shown with a power $2$ rather than a power $4$. Again, the proof there is for a Hilbert space, but holds for any other metric space.
 
The following preliminary remarks hold for any rectifiable curve $\Gamma$ in a metric space.
\begin{lemma} \label{ell-H1}
If $\Gamma$ has $\cH^1(\Gamma)<\infty$ and is connected, then there is a 1-Lipschitz function $\gamma:\bT\to \Gamma$ which is surjective.  Here, $\bT$ is a circle in $\R^2$ of circumference $32\cH^1(\Gamma)$.
\end{lemma}
For a proof, see, for example, the appendix of \cite{Schul-TSP}, where results are stated for the case of a Hilbert space there, but are valid for a compact metric space.
We will fix one such parametrization and call it $\gamma$.  We will also fix a direction of flow along $\bT$ so that we can talk about a linear ordering for any proper subarc.  We will assume without loss of generality that this is an arclength parametrization, reducing the circumference of $\bT$ if needed.


\subsection{Balls, cubes, nesting}\label{s:balls-cubes}

For parameters $C>0$ and $n_0\geq 1$, 
let $\cB$ be  a collection of balls of the form
$$\{\ball(x,C2^{-n}):x\in Y_n,\ \ n \geq n_0\}$$
where $Y_n\subset \Gamma$ is  a $2^{-n}$ separated set, {\it i.e.} $\dist(x,y)\geq 2^{-n}$ for every two distinct points $x,y\in  Y_n$.
Let $J \geq 1$ be an integer and $\kappa > 0$ be given.

We may write 
$\cB=\bigcup_{i=1}^{D'} \cB_i$, so that the collections $\cB_i\cap \cB_j=\emptyset$ if $i\neq j$, and for any $i$ and any two distinct balls $B_1,B_2\in \cB_i$ of the same radius $r$, we have $\dist(B_1,B_2)>\kappa r$.  Furthermore, for any two $B_1,B_2\in \cB_i$, we have that $\rad(B_1)/\rad(B_2) \in 2^{\jump\Z}$.
\begin{lemma}\label{l:number-of-families}
For $\H$ (or any doubling metric space for that matter), we may take $D' = D(C,\kappa)J$ where $D$ is some finite number depending only on $C$ and $\kappa$.
\end{lemma}
\begin{proof}
First, write $\cB=\cB^1 \cup....\cup \cB^J$ where $\cB^i=\{B\in \cB: \rad(B) \in C2^{i+J\bZ}\}$.
Next, write for each $i\in \{1,...,J\}$, $\cB^i=\cB^i_1\cup...\cup\cB^i_D$, where $D<\infty$ depends on $\kappa$ and $C$  and exists since $\H$ is a  doubling metric space. Thus we may take $D'=D(C,\kappa) J$.
\end{proof}

Fix a $\cB_i$ as above, and call it $\cB'$.
We will now construct a set of dyadic-like ``cubes'', one for each $B \in \cB'$, in the spirit of Christ and David  \cite{Ch,Da}.  We give the construction for one such $B \in \cB'$.  First let
$\child_0 :=\{B\}$.  For $i\geq 0$, we then set $Q_i = \bigcup_{j=0}^i \left(\bigcup \child_j\right)$ as a subset of $\H$ and write
$$\child_{i+1}:=\{B'\in \cB': B'\cap Q_i \neq \emptyset, \ \ \rad(B')\leq\rad(B)\}.$$
We let
$$Q = \bigcup_{i \geq 0} Q_i.$$
We have the following properties.

\begin{lemma} \label{cube-properties-1}
For sufficiently large $\jump \geq 100$ we have the following
\begin{enumerate}
\item 
$B\subset Q\subset (1+2^{-\jump+2})B$.
\item 
Let $Q$ and $Q'$ be two cubes that are constructed from $B$ and $B'$ of $\cB'$, respectively, as above.  If $Q \cap Q' \neq \emptyset$ and $\rad(B) > \rad(B')$, then
$Q' \subseteq Q$. 
\item
If $B_1,B_2\in\cB'$, are of the same radius $r$, then 
$\dist(Q(B_1),Q(B_2))>(\kappa-1)r$.
\end{enumerate}
\end{lemma}

\begin{proof}
Property (1): See  Lemma 2.16 in \cite{Schul-AR}. 
Property (2):  
If $Q\cap Q'\neq\emptyset$ then  one of the balls making up $Q'$ intersects $Q$.  
It follows from the construction of $Q(B)$ that any balls of radius at most $\rad(B)$ that intersect $Q(B)$ will be contained in $Q(B)$.
As $\rad(B') < \rad(B)$, all the balls making up $Q'$ will be less than $\rad(B)$.  Thus, they will eventually be absorbed into $Q(B)$ during the construction.
Property (3): follows from the similar property of $\cB'$ together with (1).
\end{proof}

We will call the resulting family of ``cubes'' $Q$ associated to balls in $\cB'$ by the name $\Delta$.
When we need to be more specific we will write $\Delta(\cB,i)$ where $i$ ranges from 1 to $D'$.  
Thus every ball $B\in \cB$ has an $i\in\{1,...,D'\}$ and $Q\in \Delta(\cB,i)$ such that $B\subset Q\subset (1+2^{-\jump+2})B$.

We will also need  a similar construction for arcs in $\gamma$, except we will also take care to get all of $\gamma$ on every scale.
\begin{lemma}\label{l:build-filtration-arcs}
Suppose $J\geq 10$ is an integer, 
{$\delta\in(0,1)$ }
$L > 0$, and $\cF^0=\bigcup\limits_{n\geq m} \cF_n^0$ is a collection of arcs in $\gamma$ such that
\begin{enumerate}[(i)]
\item 
For $\tau\in\cF_n^0$, we have $L2^{-nJ} \leq \diam(\tau) < L2^{-nJ+3}$.
\item For $\tau_1,\tau_2\in\cF_n^0$, we have $\tau_1\cap \tau_2=\emptyset$.
\item
Let $k>0$. 
If  $\zeta\in\cF_{n+k}^0$, $\tau\in \cF_n^0$ and $\zeta\cap \tau\neq\emptyset$, then $\zeta\subset \tau$.
\end{enumerate}
Then there is a collection of arcs $\cF = \bigcup_{n \geq m} \cF_n$ with the following properties
\begin{enumerate}
\item 
For $\zeta\in\cF_{n+1}$, there is a unique element $\tau\in \cF_n$ such that $\zeta\subset \tau$.
\item
For $\tau\in\cF_n$, we have $\delta L2^{-nJ}\leq \diam(\tau) < L2^{-nJ+4}$.
\item
For $\tau_1,\tau_2\in\cF_n$ we have that they are either disjoint, identical, or intersect in (one or both of) their endpoints.
\item
For all $n$,
$\bigcup\cF_n = \bT$.
\item
For each element $\tau^0\in \cF_n^0$ there is an element $\tau\in\cF_n$ such that $\tau\supset\tau^0$.
We have that domain of  $\tau\setminus \tau^0$ has at most  two connected components, each of which with  image with diameter 
$< \delta L2^{-nJ}$.
\item If $\tau^0,\tau^1\in \cF_n^0$ then they give rise to two different arcs in $\cF_n$.
\end{enumerate}
\end{lemma}

We call the families of arcs that satisfy the hypothesis and conclusion of Lemma \ref{l:build-filtration-arcs} {\it prefiltrations} and {\it filtrations} of $\bT$, respectively.

\begin{note} \label{n:j-delta-const}
We will take $\delta=2^{-10}$, which we need for the proof of Lemma \ref{l:diam-large-NEW} (any sufficiently small value would work).  
For  the proof of Proposition \ref{l:modified-prop-4} we then need to set $J=100$ (smaller values of $\delta$ would yield larger values in $J$, with $J$ depending linearly on $\log(\delta)$.)
We will also take $L = A2^l$ where $l \in \{0,...,J-1\}$.  This $L$ comes from diameter bounds of the prefiltration as given in Lemma \ref{l:prefiltration}.  The discussion following Lemma \ref{l:prefiltration}  will  be the sole place we use  Lemma \ref{l:build-filtration-arcs} to construct filtrations; the properties these filtrations will be used later in the paper.
\end{note}

\begin{remark} \label{arc-clarification}
When discussing an arc $\tau$ in $\gamma$, we are really considering the function that is the restriction $\gamma|_{I_\tau}$, 
where $I_\tau = [a(\tau),b(\tau)]$ is a closed interval in $\bT$ compatible with the chosen direction of flow.  The quantity $\diam(\tau)$ is defined to be the diameter of the image of $\tau$.  On the other hand, if we say that $\tau_1$ and $\tau_2$ intersect, or have $\tau_1\subset \tau_2$, then we are referring to the domain of these functions, i.e to a subset of $\bT$. 

Note that one immediate consequence of the diameter bounds of subarcs in the filtrations is that, for a given arc $\tau \in \cF_n$, the number of arcs $\zeta \in \cF_{n+1}$ such that $\zeta \subseteq \tau$ is finite (although there is no {\it a priori} bound).  This is because we are supposing that $\gamma$ is arclength parameterized and so a lower bound for the diameter of the image of the arc translates to a lower bound for the diameter of the domain of the arc.  This also shows that the cardinality of the all the subarcs of a filtration is countable.
\end{remark}

\begin{proof}
We construct the collections $\cF_n$ by induction. All the properties will be immediately verifiable by the construction.  As $\gamma$ is fixed, we can refer to subarcs by their domain in $\bT$ as long as we make sure to remember that their diameter is taken with respect to the image.  We start with $n=m$.  We will assume that $\cF_m^0$ does not contain the subarc that is the entire $\bT$ as otherwise we can skip ahead in $n$ until we hit such an instance.

We first suppose that $\cF_m$ contains at least two subarcs.  Let
\begin{align*}
   \bigcup_j R_{m,j} = \bT \setminus \left( \bigcup \cF_m^0 \right),
\end{align*}
where $R_{m,j}$ are disjoint open intervals.  Note that each $R_{m,j}$ is surrounded by two arcs of $\cF_m^0$.  If we have that 
{
$\diam(R_{m,j}) < \delta L2^{-mJ}$
}
(remembering that this is diameter in the image of $\gamma$), then we merge it with one of the neighboring arcs of $\cF_m^0$, choosing arbitrarily between the two, and remove it from $\{R_{m,j}\}$.  We can see that elements of the modified $\cF_m^0$ will have diameter less than $L2^{-mJ+4}$.

We now go through the remaining subarcs of $\{R_{m,j}\}$, which now all have diameter at least 
{ 
$\delta L2^{-mJ}$
}.  
If $R_{m,j}$ is a subarc such that 
{
$\delta L2^{-mJ} \leq \diam(R_{m,j}) < L2^{-mJ + 4}$,
} 
then we leave it alone.  If we get a subarc so that $\diam(R_{m,j}) \geq L2^{-mJ + 4}$, then we can partition $R_{m,j}$ into intervals of diameter between $[L2^{-mJ},L2^{-mJ+4})$ such that each element of $\cF_{m+1}^0$ is contained in a single subarc (either in $\cF_m^0$ or in one of the partitions of $R_{m,j}$).  This can be done because we have a large enough $\jump \geq 10$.  
We then let $\cF_m$ be the set composed of (possibly) extended $\cF_m^0$ and closures of the partitions of $\{R_{m,j}\}$.

In the case that $\cF_m$ contains only one subarc which is not all of $\bT$ (which we will still refer to as $\cF_m^0$ by abuse of notation), we look at its complement $R_m = \bT \setminus \cF_m^0$.  If 
{
$\diam(R_m) < \delta L2^{-mJ}$,
} 
then we merge it with $\cF_m^0$ and so $\cF_m = \{\bT\}$.  If 
{
$\delta L2^{-mJ} \leq \diam(R_m) < L2^{-mJ + 4}$,
}
then we take $\cF_m = \{\cF_m^0, \overline{R}_m\}$.  If $\diam(R_m) \geq L2^{-mJ + 4}$, then we partition it as in the previous paragraph and take $\cF_m$ to be the closures of this collection of subarcs along with $\cF_m$.

We now continue inductively.  Let $n > m$.  Let
\begin{align*}
\bigcup_j R_{n,j}=   \bT\setminus \left[\left(\bigcup\cF_{n}^0\right) \cup \left(\bigcup_{\tau\in\cF_{n-1}} \partial\tau\right)\right] 
\end{align*}
where $R_{n,j}$ are disjoint open intervals.

If we have a subarc so that 
{$\diam(R_{n,j}) < \delta L2^{-nJ}$,} 
then $R_{n,j}$ must share a boundary point with some subarc of $\cF_n^0$.  Indeed, the only other possibility is that $R_{n,j}$ has as boundary points two points of $\bigcup_{\tau \in \cF_{n-1}} \partial \tau$.  However, as $\bigcup \cF_{n-1} = \bT$ this means that there is some $\tau \in \cF_{n-1}$ so that $\tau = R_{n,j}$ and so 
{$\diam(\tau) < \delta L2^{-nJ}$}.
  This is a contradiction of the diameter bound 
 { $\diam(\tau) \geq \delta L2^{-(n-1)J}$ }
  for all $\tau \in \cF_{n-1}$.

Thus, we may, as before, merge each $R_{n,j}$ with 
{$\diam(R_{n,j}) < \delta L2^{-nJ}$ }
with one of the arcs of $\cF_n^0$ that it borders, choosing arbitrarily if there are two, and then remove it from $\{R_{n,j}\}$.  We can see that elements of the modified $\cF_n^0$ will have diameter at most $L2^{-nJ+4}$.

The remaining steps are exactly the same as before.  We go through the remaining subarcs of $\{R_{n,j}\}$, which all have diameter at least 
{$\delta L2^{-nJ}$. }   
If a subarc such that $ \diam(R_{n,j}) < L2^{-nJ + 4}$, then we leave it alone.  If we get a subarc so that $\diam(R_{n,j}) \geq L2^{-nJ + 4}$, then we can partition $R_{n,j}$ into intervals of length between $[L2^{-nJ},L2^{-nJ+4})$ such that each element of $\cF_{n+1}^0$ has a single parent (either in $\cF_n^0$ or in one of the partitions of $R_{n,j}$).  This can be done because we have a large enough $\jump \geq 10$.  We then let $\cF_n$ be the set composed of (possibly) extended $\cF_n^0$ and the closures of the subarcs making up the partitions of $\{R_{n,j}\}$.

The collection $\cF = \bigcup_n \cF_n$ is the desired filtration.
\end{proof}

\subsection{Different types of balls: flat vs. non-flat}\label{s:types-of-balls}

In this section we divide the collection of balls $\smallballs$ into different types of balls, which we will later handle by independent techniques.  However, we first need to define several families of arcs associated to every ball.

\begin{note}
Recall that  we have set $A=10$ and $\jump=100$.  We now also set $\kappa=3$.  
This value for $\kappa$ will be used 
when invoking the construction of the ``cubes" $\Delta$ and the lemma that follows it, Lemma \ref{cube-properties-1}. 
The  value of $\kappa$ could have been taken to be any number $\geq 3$ .
\end{note}
Let $2\smallballs$ denote the doubles of balls in $\smallballs$, and let $\cB=2\smallballs$.
We apply Lemma \ref{l:number-of-families} to $\cB$ with $2A = 20$ the implied constant, $\jump = 100$, and $\kappa = 3$ to get well separated subfamilies $\{\cB_i\}_{i=1}^{D'}$, where $D'=D(2A,\kappa)J=D(20,3)100$.  We then apply the construction of Lemma \ref{cube-properties-1} to produce $\{\Delta(\cB,i)\}_{i=1}^{D'}$. For each ball $B\in \smallballs$, we have thus fixed a cube $Q=Q(B)$ with 
$2B\subset Q(B)\subset (1+2^{-\jump+2})2B$.
Given such a cube $Q(B) \subset \H$ we let
\begin{align*}
  \Lambda(Q(B)) := \left\{\tau = \gamma|_{[a,b]} : \right.&[a,b] \subset \T ~\text{is a connected component of} ~\gamma^{-1}(\Gamma \cap Q) \\
  &\left.~\text{and} ~\gamma([a,b]) \cap B \neq \emptyset \right\},
\end{align*}
that is, $\Lambda(Q(B))$ composes of all connected subarcs through $Q(B)$ that intersect $B$.  See the left hand side of Figure \ref{f:tau-and-tau-prime}.  For each $i \in \{1,...,D'\}$, let $\cF^{0,i} = \bigcup_{Q(B) \in \Delta(\cB,i)} \Lambda(Q(B))$.

\begin{lemma} \label{l:prefiltration}
  For each $i \in \{1,...,D'\}$, $\cF^{0,i}$ is a prefiltration and there exists some $l(i) \in \{0,...,J-1\}$ such that we have the diameter bounds
  \begin{align}
    A2^{l(i)} 2^{-kJ} \leq \diam(\tau) < A 2^{l(i)} 2^{-kJ + 3}, \qquad \forall \tau \in \cF^{0,i}_k. \label{e:prefiltration-diam-bound}
  \end{align}
\end{lemma}

\begin{proof}
  Let $i \in \{1,...,D'\}$ be fixed, choose some $B \in \cB_i$, and let $\tau \in \Lambda(Q(B))$.  Remembering that $2B \subseteq Q(B) \subseteq (1+2^{-J+2})2B$ and that $\tau(I_\tau) \cap B \neq \emptyset$, we get that
  \begin{align}
  5\rad(B)\geq \diam(\tau)\geq \rad(B)\,.  \label{e:tauQ-rad-bound}
  \end{align}
  One of the properties of $\cB_i$ is that there exists some $l \in \{0,...,J-1\}$ so that $r(B) \in A2^{l + J\Z}$.  Thus, it is clear that $\cF^{0,i}$ can be decomposed as a collection of curves $\bigcup_j \cF^{0,i}_j$ that satisfies $\eqref{e:prefiltration-diam-bound}$.
 Thus we have property (i) of a prefiltration.

Since $\kappa=3$, we have property (ii) of a prefiltration from Lemma \ref{cube-properties-1} (3).

  Now suppose $k > 0$ and $\tau \in \cF_n^{0,i}$, $\tau' \in \cF_{n+k}^{0,i}$ such that $\tau \in \Lambda(Q(B))$, $\tau' \in \Lambda(Q(B'))$, and $\tau \cap \tau' \neq \emptyset$ (remembering how we defined two arcs intersecting in Remark \ref{arc-clarification}).  Thus, $Q(B) \cap Q(B') \neq \emptyset$.  As $\rad(B) > \rad(B')$, we get that $Q(B') \subset Q(B)$ and so $\tau' \subset \tau$.
   Thus we have property (iii) of a prefiltration.

\end{proof}

By Lemma \ref{l:build-filtration-arcs} applied with \eqref{e:prefiltration-diam-bound} and $L = A2^{l(i)}$, we can complete each $\cF^{0,i}$ to a filtration $\cF^i$.  Thus, for each $\tau \in \cF^{0,i}_k$, there exists some $\tau' \in \cF^i_k$ such that $\tau \subseteq \tau'$.  We then define for each $B \in \smallballs$
\begin{align*}
  \Lambda'(Q(B)) := \{ \tau' : \tau \in \Lambda(Q(B)) \}.
\end{align*}
See the right hand side of Figure \ref{f:tau-and-tau-prime}.

\begin{figure}
\scalebox{0.8}{
\scalebox{0.5}{\includegraphics*{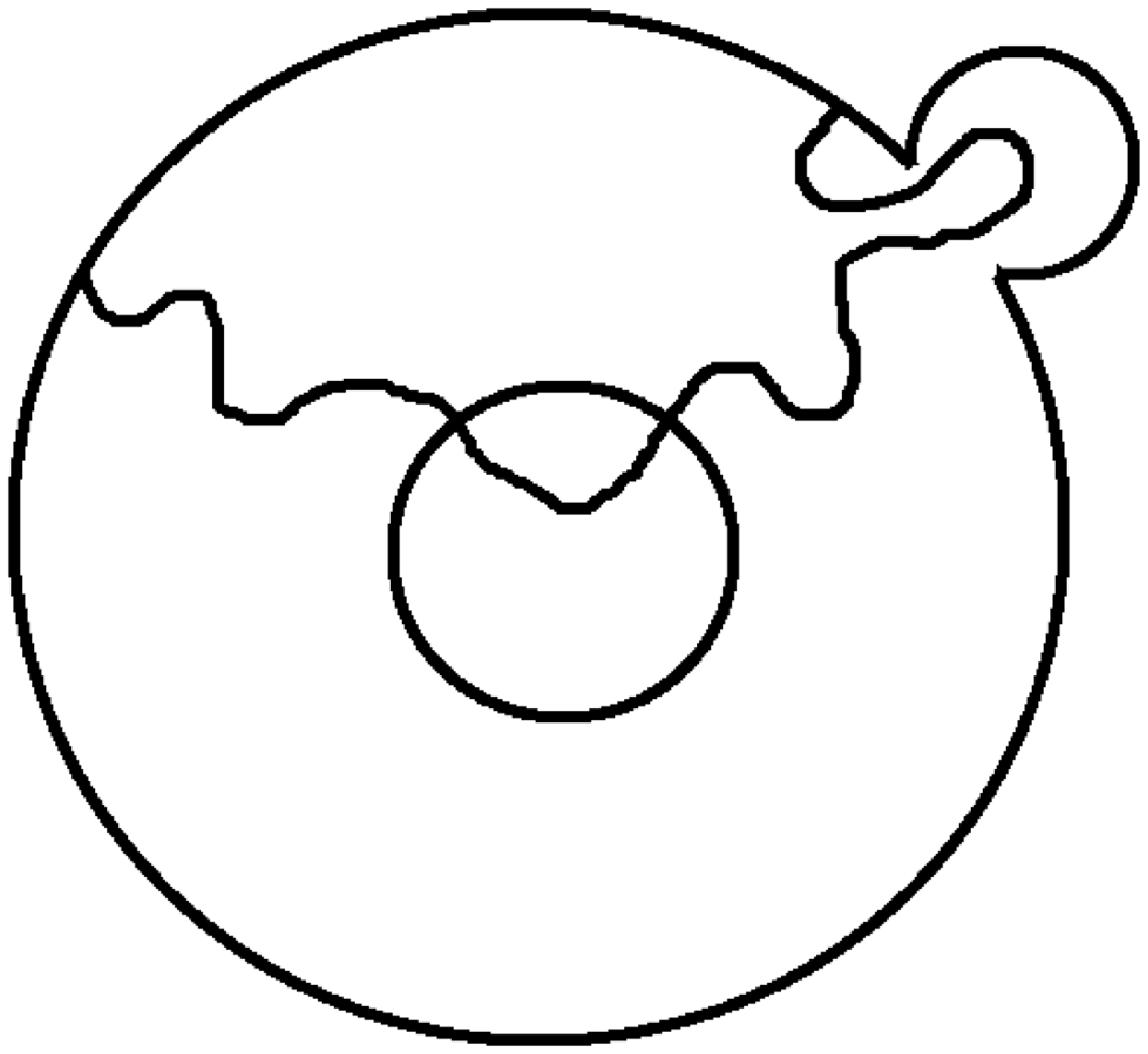}}
\setlength{\unitlength}{1mm}
\begin{picture}(0,0)
\put(-33,65){ $Q$}
\put(-65,42){ $2B$}
\put(-65,54){ $B$}
\put(-80,80){   $\tau$}
 \end{picture}
 
 \hspace{-1.5in}
 
\scalebox{0.5}{\includegraphics*{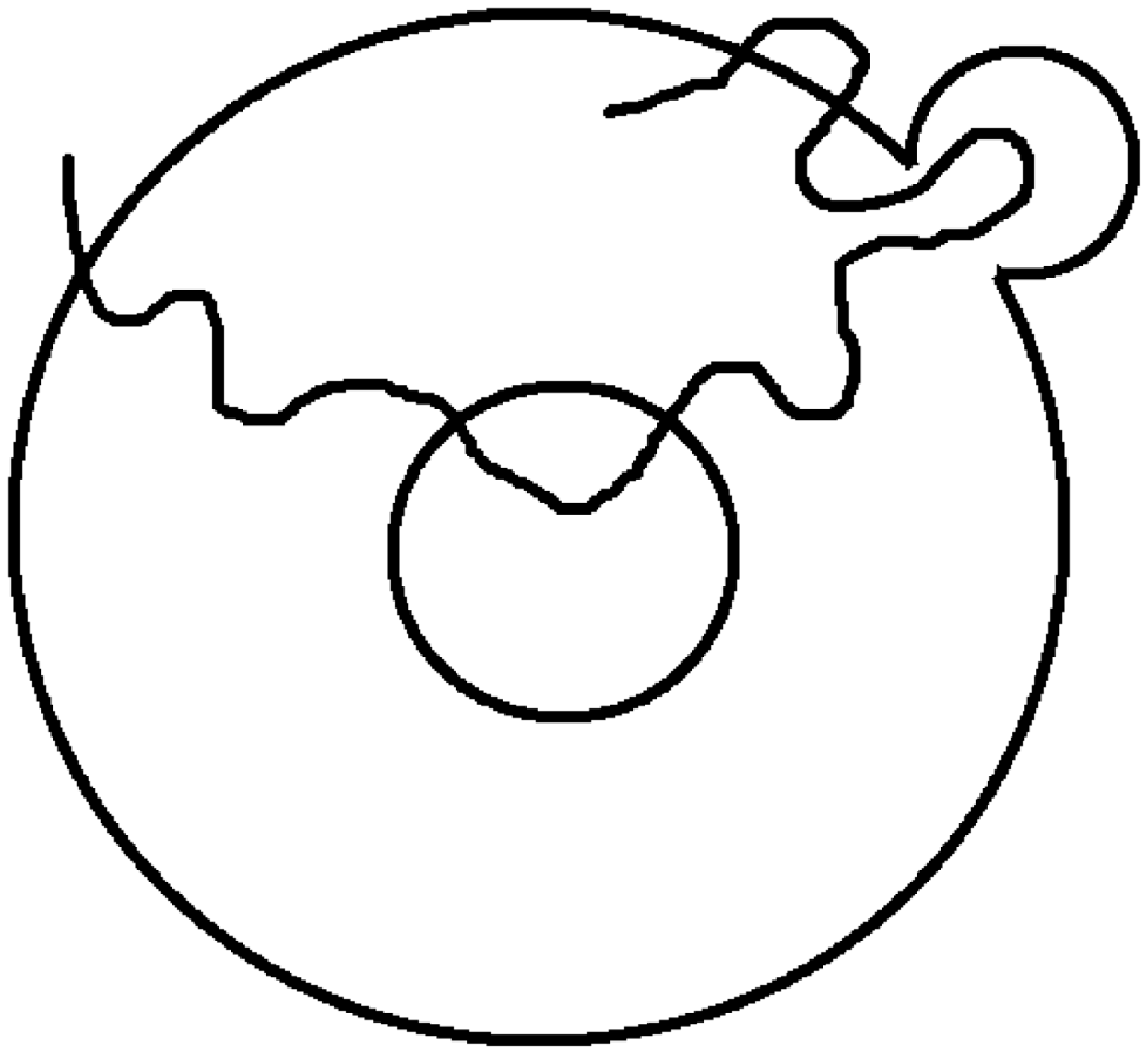}}
\setlength{\unitlength}{1mm}
\begin{picture}(0,0)
\put(-33,65){ $Q$}
\put(-65,42){ $2B$}
\put(-65,54){ $B$}
\put(-80,80){   $\tau'$}
 \end{picture}
 }
 
  \vspace{-1in}

\caption{Left:  $\tau\in \Lambda(Q)$. Right: $\tau'\in \Lambda'(Q)$ \label{f:tau-and-tau-prime}}
\end{figure}


Given some subarc $\tau$, we can define
\begin{align*}
  L_\tau := \{\gamma(a(\tau)) \delta_t \tilde{\pi}(\gamma(a(\tau))^{-1}\gamma(b(\tau))) : t \in [0,1]\},
\end{align*}
that is, $L_\tau$ is the horizontal line segment that starts from $\gamma(a(\tau))$ and goes horizontally towards $\gamma(b(\tau))$, possible, without hitting $\gamma(b(\tau))$; see Remark \ref{r:doesnt-hit}.  We can then define the quantity
\begin{align*}
  \beta(\tau) := \sup_{t \in I_\tau} \frac{d(\gamma(t),L_\tau)}{\diam(\tau)}.
\end{align*}
Thus, $\beta(\tau)$ evaluates how far $\tau$ can get from the specific horizontal line segment $L_\tau$.  Recall that $\diam(\tau)$ is measured with respect to its image.

\begin{note}\label{n:epsilon-s}
We fix $\epsilonSubS=10^{-10}$ (any sufficiently small constant would suffice).
The value of this constant will become apparent in Section \ref{s:flat-balls}; the first time its value is used is in Lemma \ref{l:large-subsarc}. 
\end{note}
We let 
\begin{align*}
\flatballs:=\{B\in\smallballs: \beta(\tau) < \epsilonSubS \beta_\Gamma(B) ~\forall \tau \in \Lambda'(Q(B)) \} 
\end{align*}
and let
\begin{align*}
  \curvyballs:=\smallballs\setminus\flatballs = \{ B \in \smallballs : \exists \tau \in \Lambda'(Q(B)) ~\text{such that} ~\beta(\tau) \geq \epsilonSubS \beta_\Gamma(B)\}.
\end{align*}
%
We will show in Section \ref{s:curvy-balls} that
 \begin{align}
\sum\limits_{B\in \curvyballs} \beta_\Gamma(B)^4\diam(B)\leq C\cH^1(\Gamma)\,, \label{e:G1-control}
\end{align}
and we will show in Section \ref{s:flat-balls} that
 \begin{align}\label{e:G2-control}
\sum\limits_{B\in \flatballs} \beta_\Gamma(B)^4\diam(B)\leq C\cH^1(\Gamma)\,.
\end{align}

\section{Non-flat balls}\label{s:curvy-balls}
In this section we prove \eqref{e:G1-control}.  Recall that we have a fixed parametrization $\gamma$ (see the discussion after Lemma \ref{ell-H1}). Also recall from Lemma \ref{l:build-filtration-arcs} that a filtration is constructed from a prefiltration with parameters $\jump$, $\delta$, $L$, and $m$.  The primary result that we will use to prove \eqref{e:G1-control} is the following proposition.

\begin{proposition} \label{p:beta-filtration}
  For any filtration $\cF$ constructed with $\jump=100$ and $\delta = 2^{-10}$ ($m$ and $L$ are allowed to be arbitrary), we have
  \begin{align}\label{e:prop-beta-filtration-statement}
    \sum\limits_{\tau\in\cF} \beta(\tau)^4\diam(\tau)\leq \frac{10^{14}2^{4\jump+66}}{\eta^2} \ell(\gamma).
  \end{align}
\end{proposition}
\begin{note}
One may consider stating the above proposition for  $J\geq 100$ and $\delta\in (2^{-10},1)$ which would suffice for Lemma \ref{l:modified-prop-4}.
One may also consider varying $\delta$ in the range $\delta\in (2^{-J-6},1)$, however then the constant on the right hand side of  \eqref{e:prop-beta-filtration-statement} would need to be modified.
 {\bf  An important point} is that  if one does any of these, then 
one would also need to modify the $\eta$ which has already been  fixed  after  Proposition \ref{p:prop-4} (see Note \ref{n:n2.4}).
 It is for this reason, that we fix specific values for $J$ and $\delta$.
\end{note}
Before we prove the proposition, we first use it to prove the following corollary, which proves \eqref{e:G1-control}.  Recall how $\curvyballs$ was constructed in Section \ref{s:types-of-balls}.

\begin{cor}
  With the choices of parameters $\kappa=3$, $\jump=100$, and $A=10$, $\eta=2^{-1200}$,  $\delta = 2^{-10}$ and  $\epsilonSubS=10^{-10}$, 
there exists some absolute constant $C > 0$ such that
  \begin{align*}
    \sum\limits_{B\in \curvyballs} \beta_\Gamma(B)^4\diam(B)\leq C\cH^1(\Gamma).
  \end{align*}
\end{cor}

\begin{proof}[Proof of corollary]
  Note that the partition of $\cB(=2\smallballs)$ into $D'$ separated subfamilies $\{\cB_i\}_{i=1}^{D'}$ by Lemma \ref{l:number-of-families} also partitions $\curvyballs$ (really 2$\curvyballs$) into $D'$ separated subfamilies, which we will refer to as $\{\smallballs^i\}_{i=1}^{D'}$.  We remind the reader that $D'$ is a constant depending only on $\kappa=3$, $\jump=100$, and $2A=20$.

  By definition, for each $i$ and each $B \in \smallballs^i$ there exists some $\tau_B \in \Lambda'(Q(B))$ such that
  \begin{align}
    \beta_\Gamma(B) \leq \frac{1}{\epsilonSubS} \beta(\tau_B). \label{e:G1-redefn}
  \end{align}
  By construction, for each $i$, all elements of $\Lambda'(Q(B))$ for all $B \in \smallballs^i$ were subarcs taken from one specific filtration $\cF^i$ (out of $D'$ possible filtrations).  In addition, by Lemma \ref{l:build-filtration-arcs}
  we have that each  $\tau_B$ corresponds to a unique subarc of $\cF^i$. Thus, we have by Proposition \ref{p:beta-filtration} that
  \begin{multline*}
    \sum_{B \in \smallballs_1} \beta_\Gamma(B)^4 \diam(B) = \sum_{i=1}^{D'} \sum_{B \in \smallballs^i} \beta_\Gamma(B)^4 \diam(B) \overset{\eqref{e:tauQ-rad-bound} \wedge \eqref{e:G1-redefn}}{\leq} \frac{2}{\epsilonSubS^4} \sum_{i=1}^{D'} \sum\limits_{B\in \smallballs^i} \beta(\tau_B)^4\diam(\tau_B) \\
    \leq \frac{2}{\epsilonSubS^4} \sum_{i=1}^{D'} \sum_{\tau \in \cF^i} \beta(\tau)^4 \diam(\tau) \leq \frac{10^{14}2^{4\jump+67}D'}{\epsilonSubS^4\eta^2} \ell(\gamma) \leq \frac{10^{14}2^{4\jump+73}D'}{\epsilonSubS^4\eta^2} \cH^1(\Gamma).
  \end{multline*}
  In the last inequality, we used the fact that $\ell(\gamma) \leq 32\cH^1(\Gamma)$, which can be easily be seen from Lemma \ref{ell-H1}.
\end{proof}

Note that the proposition holds true {\it a posteriori} for any metric on $\H$ that is biLipschitz with $d$ (in particular, the Carnot-Carath\'{e}odory metric), although the multiplicative constant in the inequality will depend on the biLipschitz distortion.  Thus, so does the corollary.

Thus, it remains to prove the proposition.  We now let $\cF$ be some filtration satisfying the hypotheses of Proposition \ref{p:beta-filtration} that we fix for the rest of the section.  
We will need an improved  version of Proposition \ref{p:prop-4}.  Before we state it, we establish some notation.  For $\tau \in \cF_n$ and $k \in \N$, we let
\begin{align*}
  \cF_{\tau,k} := \{\tau' \in \cF_{n+k} : \tau' \subset \tau\}.
\end{align*}
We can now define
\begin{align*}
  d_\tau = \max_{\tau' \in \cF_{\tau,1}} \sup_{z \in L_{\tau'}} d\left(z,L_\tau\right)
\end{align*}
to be the maximal distance from the discontinuous piecewise-horizontal polygonal line determined by the endpoints of $\cF_{\tau,1}$ and $L_\tau$.  Keep in mind that we have fixed an orientation of $\bT$ so that $a$ and $b$, the endpoint functions, are uniquely determined.  We first prove the following lemma.

\begin{lemma} \label{l:dtau-diam-bound}
  $d_\tau \leq 2\diam(\tau)$.
\end{lemma}

\begin{proof}
  Let $\tau' \in \cF_{\tau,1}$ and $z \in L_{\tau'}$.  As $N(\tilde{\pi}(g)) \leq N(g)$, we have that
  \begin{align*}
    d(z,\gamma(a(\tau'))) \leq \diam(\tau').
  \end{align*}
  Thus,
  \begin{multline*}
    d(z,L_\tau) \leq d(z,\gamma(a(\tau))) \leq d(z,\gamma(a(\tau'))) + d(\gamma(a(\tau')),\gamma(a(\tau))) \leq \diam(\tau') + \diam(\tau) \\
    \leq 2\diam(\tau).
  \end{multline*}
\end{proof}

We can now state our improved version of Proposition \ref{p:prop-4}.

\begin{lemma}\label{l:modified-prop-4}
  For any $\tau \in \cF$, we have that
  \begin{align}
    \frac{d_\tau^4}{\diam(\tau)^3} \leq \frac{10^{14}2^{4\jump+64}}{ \eta^2} \left(\left( \sum_{\tau' \in \cF_{\tau,2}} d(\gamma(a(\tau')),\gamma(b(\tau')))\right) - d(\gamma(a(\tau)),\gamma(b(\tau))) \right). \label{dtau-triangle-bound}
  \end{align}
\end{lemma}

\begin{proof}
  We let $\tau \in \cF_k$.  We first suppose that 
  \begin{align*}
    \left( \sum_{\tau' \in \cF_{\tau,2}} d(\gamma(a(\tau')),\gamma(b(\tau')))\right) - d(\gamma(a(\tau)),\gamma(b(\tau)))  \geq \delta L 2^{-\jump-3} 2^{-k\jump}.
  \end{align*}
  By the properties of the filtration and Lemma \ref{l:dtau-diam-bound}, we have
  \begin{align*}
    \frac{d_\tau^4}{\diam(\tau)^3} \leq 16\diam(\tau) \leq L2^{-k\jump+6}.
  \end{align*}
  We then get \eqref{dtau-triangle-bound}.  Thus, we may assume that
  \begin{align}
    \left( \sum_{\tau' \in \cF_{\tau,2}} d(\gamma(a(\tau')),\gamma(b(\tau'))) \right) - d(\gamma(a(\tau)),\gamma(b(\tau))) < \delta L 2^{-\jump-3} 2^{-k\jump}. \label{excess-small}
  \end{align}
  Let $\{\tau_i\}_{i=1}^m$ denote the subarcs of $\cF_{\tau,1}$ in order as denoted by the flow along $\bT$ (thus, $\gamma(a(\tau_1)) = \gamma(a(\tau))$, $\gamma(a(\tau_{i+1})) = \gamma(b(\tau_i))$, and $\gamma(b(\tau_m)) = \gamma(b(\tau))$).
  
  We define
  \begin{align*}
    \mathcal{P} := \bigcup_{i=1}^{m-1} \{\gamma(b(\tau_i))\},
  \end{align*}
  and we claim that
  \begin{align}
    d(\mathcal{P},\{\gamma(a(\tau)),\gamma(b(\tau))\}) \geq \delta L2^{-J-2} 2^{-k\jump}. \label{children-well-separated}
  \end{align}
  Indeed, suppose not.  Then there exists some point $z \in \mathcal{P}$ so that, say, $d(z,\gamma(a(\tau))) < \delta L 2^{-(k+1)\jump-2}$.  Let $\xi$ denote the subarc with endpoints $\gamma(a(\tau))$ and $z$.  Then $\xi$ contains some subarc of $\cF_{\tau,1}$
  and by the property of filtrations, we must have that
  \begin{align*}
    \diam\left(\xi\right) \geq \delta L2^{-(k+1)\jump}.
  \end{align*}
  Thus, there exists a point $w \in \xi$ so that
  \begin{align*}
    d(w,\{\gamma(a(\tau)),z\}) \geq \delta L 2^{-(k+1)\jump-2}.
  \end{align*}
  As the filtration covers all of $\bT$, there must exist some $\tilde{\tau} \in \cF_{\tau,2}$ so that $w \in \tilde{\tau}$.  We get by the triangle inequality, and the fact that $\diam(\tilde{\tau}) \leq L 2^{-(k+2)\jump+4}$, that
  \begin{align*}
    d(\gamma(a(\tilde{\tau})),\{\gamma(a(\tau)),z\}) \geq \delta L 2^{-(k+1)\jump-2} - L 2^{-(k+2)\jump+4} \geq \delta L 2^{-(k+1)\jump-3}.
  \end{align*}
  In the last inequality, we used the fact that $\jump = 100$ and $\delta = 2^{-10}$.  Now we have by repeated use of the triangle inequality that
  \begin{align*}
    \sum_{\tau' \in \cF_{\tau,2}} &d(\gamma(a(\tau')),\gamma(b(\tau'))) - d(\gamma(a(\tau)),\gamma(b(\tau))) \\
    &\geq d(\gamma(a(\tau)),\gamma(a(\tilde{\tau}))) + d(\gamma(a(\tilde{\tau})),z) + d(z,\gamma(b(\tau))) - d(\gamma(a(\tau)),\gamma(b(\tau))) \\
    &\geq \delta L 2^{-(k+1)\jump-3} + d(\gamma(a(\tau)),z) + d(z,\gamma(b(\tau))) - d(\gamma(a(\tau)),\gamma(b(\tau))) \\
    &\geq \delta L 2^{-(k+1)\jump-3},
  \end{align*}
  which is a contradiction of \eqref{excess-small}.  Thus, we may now assume \eqref{children-well-separated}.  This then gives the inequality
  \begin{align}
    d(P,\{\gamma(a(\tau)),\gamma(b(\tau))\}) \geq \delta 2^{-\jump-6} L 2^{-k\jump+4} \geq 2^{-\jump-16} d(\gamma(a(\tau)),\gamma(b(\tau))). \label{separation-assumption}
  \end{align}

  Let $i \in \{2,...,m-1\}$.  Using \eqref{separation-assumption} and Proposition \ref{p:prop-4} with $\epsilon = 2^{-\jump-16}$ and our choice of  $\eta < 2^{-10\jump-160}/10^{10}$, 
   we get that
  \begin{align*}
    \sup_{z \in L_{\tau_i}} &\frac{d(z,L_\tau)^4}{\diam(\tau)^3} \\
    &\leq 
    	 \frac{10^{14}2^{4\jump+64}}{ \eta^2} \left( d(\gamma(a(\tau)),\gamma(a(\tau_i))) + d(\gamma(a(\tau_i)),\gamma(b(\tau_i))) \right.\\
    &\qquad \left.+ d(\gamma(b(\tau_i)),\gamma(b(\tau))) -   d(\gamma(a(\tau)),\gamma(b(\tau)))\right) \\
    &\leq \frac{10^{14}2^{4\jump+64}}{ \eta^2} \left( \left( \sum_{\tau' \in \cF_{\tau,1}} d(\gamma(a(\tau_i)),\gamma(b(\tau_i)))\right) - d(\gamma(a(\tau)),\gamma(b(\tau))) \right) \\
    &\leq \frac{10^{14}2^{4\jump+64}}{ \eta^2} \left(  \left(\sum_{\tau' \in \cF_{\tau,2}} d(\gamma(a(\tau_i)),\gamma(b(\tau_i)))\right) - d(\gamma(a(\tau)),\gamma(b(\tau))) \right).
  \end{align*}
 To get the same bounds for $\tau_1$, apply  Proposition \ref{p:prop-4}  with $p_1=\gamma(a(\tau))$, $p_2=p_3=\gamma(b(\tau_1))$, and $p_4=\gamma(b(\tau))$.
 Similarly, for $\tau_m$.
\end{proof}

Given any arc of a filtration $\tau \in \cF$, we can define a sequence of subarcs intervals $\{\tau_j\}_{j=0}^\infty$ so that $\tau_0 = \tau$ and $\tau_j \in \cF_{\tau,j}$ is chosen so that $d_{\tau_j}$ is maximal among all subintervals of $\cF_{\tau,j}$.

\begin{lemma}
  Let $\tau \in \cF$.  Then
  \begin{align}
    \beta(\tau)\diam(\tau) \leq \sum_{k=0}^\infty d_{\tau_k}. \label{sum-df-lower-bound}
  \end{align}
\end{lemma}

\begin{proof}
  We recursively choose a sequence of intervals $\zeta_0 = \tau$ and $\zeta_{k+1} \in \cF_{\zeta_k,1}$ so that $\beta(\zeta_{k+1}) \diam(\zeta_{k+1})$ is maximal of all possible values.  It suffices to prove that
  \begin{align*}
    \beta(\zeta_k) \diam(\zeta_k) \leq \beta(\zeta_{k+1}) \diam(\zeta_{k+1}) + d_{\tau_k}.
  \end{align*}
  Indeed, as $\beta(\zeta_k) \leq 2$ when $\gamma$ is 1-Lipschitz, we have that $\beta(\zeta_k) \diam(\zeta_k) \leq 2\diam(\zeta_k) \to 0$.  Thus,
  \begin{align*}
    \sum_{k=0}^\infty d_{\tau_k} \geq \sum_{k=0}^\infty (\beta(\zeta_k) \diam(\zeta_k) - \beta(\zeta_{k+1}) \diam(\zeta_{k+1})) = \beta(\tau) \diam(\tau).
  \end{align*}
  We can bound
  \begin{align*}
    \beta(\zeta_k) \diam(\zeta_k) &= \sup_{z \in \zeta_k} d\left(z,L_{\zeta_k}\right) \\
    &\leq \max_{\tau' \in \cF_{\zeta_k,1}} \sup_{z \in \tau'} d(z,L_{\tau'}) + \max_{\tau' \in \cF_{\zeta_k,1}} \sup_{z \in L_{\tau'}} d\left(z,L_\tau\right) \\
    &\leq \beta(\zeta_{k+1}) \diam(\zeta_{k+1}) + d_{\tau_k}.
  \end{align*}
\end{proof}

We can now prove Proposition \ref{p:beta-filtration}.

\begin{proof}[Proof of Proposition \ref{p:beta-filtration}]
  By our choice of $\eta$, Lemma \ref{l:modified-prop-4} shows that,
  \begin{align*}
    \sum_{\tau \in \cF_n} \frac{d_{\tau}^4}{\diam(\tau)^3} \leq \frac{10^{14}2^{4\jump+64}}{ \eta^2} \left( \sum_{\tau \in \cF_{n+2}} d(\gamma(a(\tau)),\gamma(b(\tau))) - \sum_{\tau \in \cF_n} d(\gamma(a(\tau)),\gamma(b(\tau)))\right).
  \end{align*}
  Summing over $n$ we get that
  \begin{align}
    \sum_{\tau \in \cF} \frac{d_{\tau}^4}{\diam(\tau)^3} \leq \frac{10^{14}2^{4\jump+65}}{ \eta^2}  \sup_{n \in \N} \sum_{\tau \in \cF_n} d(\gamma(a(\tau)),\gamma(b(\tau))) \leq \frac{10^{14}2^{4\jump+65}}{ \eta^2}  \ell(\gamma). \label{e:sum-dtau-bound}
  \end{align}
  We can now compute in an $\ell_4$ fashion:
  \begin{align*}
    \left(\sum_{\tau \in \cF} \beta(\tau)^4 \diam(\tau)\right)^{1/4} &\overset{\eqref{sum-df-lower-bound}}{\leq} \left(\sum_{\tau \in \cF} \frac{\left( \sum_{k=0}^\infty d_{\tau_k} \right)^4}{\diam(\tau)^3} \right)^{1/4} \\
    &\leq \sum_{k=0}^\infty \left( \sum_{\tau \in \cF} \frac{d_{\tau_k}^4}{\diam(\tau)^3} \right)^{1/4} \\
    &\leq \sum_{k=0}^\infty 2^{-3(\jump+1) k/4}\left( \sum_{\tau \in \cF} \frac{d_{\tau_k}^4}{2^{-3(\jump+1) k}\diam(\tau)^3} \right)^{1/4} \\
    &\leq \sum_{k=0}^\infty 2^{-3(\jump+1) k/4}\left( \sum_{\tau \in \cF} \frac{d_{\tau_k}^4}{\diam(\tau_k)^3} \right)^{1/4} \\
    &\overset{\eqref{e:sum-dtau-bound}}{\leq} \frac{10^{14}2^{4\jump+65}}{ \eta^2}  \sum_{k=0}^\infty 2^{-3(J+1)k/4} \ell(\gamma)^{1/4} \\
    &\leq \frac{10^{14}2^{4\jump+66}}{ \eta^2}  \ell(\gamma)^{1/4}.
  \end{align*}
  In the last inequality, we used the fact that $\jump = 100$ to show that $\sum 2^{-3(\jump+1)k/4} < 2$.
\end{proof}

\section{Flat balls}\label{s:flat-balls}

\subsection{Geometric lemmas about arcs}

The following lemma states that if an arc $\tau$
is close to the horizontal line segment interpolating its endpoints, then this horizontal line segment is also close to 
$\tau$
all throughout.

\begin{lemma}\label{l:lemma-8}
  Let $\tau$ be a connected subarc.  Then
  \begin{align}
    \sup_{x \in L_\tau} d(x,\tau) \leq \beta(\tau) \diam(\tau). \label{reverse-close}
  \end{align}
Moreover, we have that the start-point of $L_\tau$ is the same as $\gamma(a_\tau)$, and the end-point of $L_\tau$ has distance at most $ \beta(\tau) \diam(\tau)$ to 
$\gamma(b_\tau)$.
 \end{lemma}
\begin{proof}
  By translation and rotation, we may suppose without loss of generality that the endpoints of $L_\tau$ are $(0,0,0)$ and $(l,0,0)$ and such that $\gamma(a(\tau)) = (0,0,0)$.  Consider the closed set
  \begin{align*}
    F := \left\{ (t,z) \in [a(\tau),b(\tau)] \times L_\tau : d(\gamma(t),z) \leq \beta(\tau) \diam(\tau) \right\}.
  \end{align*}
  It suffices to show that the projection of $F$ to the second factor is all of $L_\tau$.

  As $\gamma(a(\tau))$ is an endpoint of $L_\tau$, it follows that $F \cap ([a(\tau),b(\tau)] \times \{(0,0,0)\})$ is nonempty.  The other endpoint $(l,0,0)$ also satisfies $F \cap ([a(\tau),b(\tau)] \times \{(l,0,0)\}) \neq \emptyset$.  Indeed, we must have that $\gamma(b(\tau)) = (l,0,z)$ for some $z \in \R$.  Thus, it follows that
  \begin{align*}
    \beta(\tau) \diam(\tau) \geq d(\gamma(b(\tau)),L_\tau) = d(\gamma(b(\tau)),(l,0,0)).
  \end{align*}
  
  As $d(\gamma(t),L_\tau) \leq \beta(\tau) \diam(\tau)$ for all $t \in [a(\tau),b(\tau)]$, we get for all $t \in [a(\tau),b(\tau)]$ that $F \cap (\{t\} \times L_\tau) \neq \emptyset$.  In addition, as balls of the Koranyi metric are convex subsets of $\R^3$ (balls at the origin are convex and Heisenberg translations are affine) and $L_\tau$ is also an affine line segment, we get that $F \cap (\{t\} \times L_\tau)$ is a connected interval.

  We are now in the following situation: $F$ is a closed subset of a rectangle that intersects each vertical slice in an interval as well as intersecting the top and bottom sides, and we would like to show that $F$ intersects each horizontal slice.  To do so, it clearly suffices to prove that $F$ is connected.

  Suppose $F$ is not connected.  Then there exists a continuous surjection $f : F \to \{0,1\}$.  As $F$ intersects each vertical slice in a conected set, we have that $f$ is constant on vertical slices.  Thus, we may define a function $g : [a(\tau),b(\tau)] \to \{0,1\}$ by $g(t) = f(t,z)$ for $(t,z) \in F$.  This function $g$ is continuous because $F$ is closed.  As $[a(\tau),b(\tau)]$ is connected, $g$ must be constant.  Then $f$ must be constant, which is a contradiction of its surjectivity.  Thus, $F$ is connected, which finishes the proof.
\end{proof}

\begin{remark}\label{r:chord-arc}
The remainder of this section relies on the above lemma and two facts.  The first fact is that, for any ball $B \subset \H$ and any $\lambda > 1$, we have
\begin{align}
  \diam(\lambda B) \leq \lambda \diam(B). \label{e:diam-scale}
\end{align}
The second fact is that for a horizontal line (segment) $L:[0,T]\to \H$ we have a constant $C_\H$ 
\begin{align}\label{e:chord-arc}
C_\H^{-1} |t_1-t_2|\leq \dist(L(t_1), L(t_2))\leq C_\H|t_1-t_2|\,. 
\end{align}
Indeed, this holds with $C_\H=1$ as $L$ is isometric to the Euclidean interval $[0,T]$.
The above lemma and these  facts will be the only properties of $\H$ that we will use.  Otherwise, it is a purely metric section i.e. the results within it hold in any metric space. Below, we make use of  the fact that $C_\H=1$ and omit the constant, otherwise $C_\H$ would have appeared in eq. \eqref{e:beta-small-for-chord-arc} and its derivatives.
\end{remark}

\begin{lemma}\label{l:diam-large-NEW}
Let $B\in\smallballs$ 
be a ball of radius $r$.
Let $Q=Q(B)$, and in particular
suppose $3B\supset Q \supset 2B$.
Suppose 
$\tau'\in \Lambda'(Q)$
and $\tau'\ni\cent(B)$.  
Suppose further that
\begin{align}\label{e:beta-small-for-chord-arc}
\beta(\tau')\diam(\tau')<h<\frac{1}{10}r 
\end{align}
Then 
there is an arc $\tilde\tau \subset \tau'$ with image in $2B$ such that $\diam(\tilde\tau)\geq 4r-10h$
\end{lemma}

\begin{proof}
Let  $L=L_{\tau'}$ and 
$C(L,h)=\{P\in\H : \dist(P,L)<h\}$.
By our assumption, 
for all image points $t$ of $\tau'$ we have $t\in C(L,h)$.

By definition, we know that $\tau'$ is an extension of an arc $\tau\in \Lambda(Q)$.  
Since $\gamma(a_\tau), \gamma(b_\tau) \in \partial Q$, and using  Lemma \ref{l:build-filtration-arcs} (and say, $\delta<1/10$) we deduce that $\tau\ni\cent(B)$. The arc $\tilde\tau$ will eventually be a sub arc of $\tau$.  We argue its existence as follows.

First, note that 
$$\dist(\gamma(a_\tau),\cent(B))\geq 2r,\ \ \ \ \ \ \dist(\gamma(b_\tau),\cent(B))\geq 2r.$$   
This implies that $\diam(\tau)\geq 2r$, 
which by Lemma \ref{l:lemma-8} implies that 
$\diam(L_{\tau'})\geq 2r-2h$. 
Using \eqref{e:chord-arc}, the second part of the statement of   Lemma \ref{l:lemma-8}, 
and that $L_{\tau'}$ starts at $\gamma(a_{\tau'})$, 
we have that
$$\dist(\gamma(a_{\tau'}),\gamma(b_{\tau'}))\geq 2r-3h\,.$$
Using Lemma \ref{l:build-filtration-arcs},
we have
$$\dist(\gamma(a_{\tau}),\gamma(b_{\tau}))\geq (2r-3h)-2\delta\diam(\tau) = 2r-3h-12r\delta \,.$$
Using $\delta\leq 1/100$, we get 
$$\dist(\gamma(a_{\tau}),\gamma(b_{\tau}))> \frac32 r.$$
Let $x, z$ be the closest points on $L$ to $\gamma(a(\tau)), \gamma(b(\tau))$ respectively. 
Let $y$ be the closest point on $L$ to $\cent(B)$. 
We will show
\begin{equation}\label{e:x-and-y-far}
\dist(x,z) \geq  4r-4h. 
\end{equation}
Indeed, 
$\gamma(a_{\tau}),\gamma(b_{\tau})\in\partial Q\cap C(L,h)$
and so, using 
Remark \ref{r:chord-arc} we deduce that 
$\gamma(a_{\tau}),\gamma(b_{\tau})$ are in different components of $C(L,h)\setminus \frac12 B$.
Thus, if we consider the order given by $L$, we have $x<y<z$ and so
$\dist(x,z) =\dist(x,y)+ \dist(y,z) \geq 2r-2h + 2r-2h= 4r-4h$, giving \eqref{e:x-and-y-far} . 

We have that $\tau$ connects between the  balls $\ball(x,h)$ and $\ball(z,h)$.  
In particular, there is a subarc of $\tau$ connecting $\ball(x,\frac32 h)$ and $\ball(z, \frac32 h)$ which does not leave $2B$:  this follows from the fact that $C(L,h)$ contains the image of $\tau$ and each of the spheres 
$\partial \ball(x,\frac32 h)$ and $\partial \ball(z,\frac32 h)$ disconnects $C(L,h)$.
Call such  an arc $\tau_1$.

Then, 
$\dist(\gamma(a(\tau_1)),x)<\frac32 h$, and $\dist(\gamma(b(\tau_1)), z)<\frac32 h$. 
In total we have 
$$\diam(\tau_1)\geq 
\dist(\gamma(a(\tau_1)),\gamma(b(\tau_1)))
\geq \dist(x,z) -3h
\geq 4r-7h\,.$$
Take $\tilde\tau=\tau_1$.
\end{proof}
%


Recall that we have fixed  $\epsilonSubS=10^{-10}$ (see Note \ref{n:epsilon-s}). This part of the paper is where we start to use this value.  As is evident below, any sufficiently small constant would have sufficed.

\begin{lemma}\label{l:large-subsarc}
Let $B\in \flatballs$ be a ball of radius $r$ and $Q=Q(B)$.
If
$\xi,\tau\in \Lambda(Q)$, 
$\tau\ni\cent(B)$
  and $\tau$ has extension to $\tau'\in\Lambda'(Q)$ 
such that $\beta(\tau')< \epsilonSubS\beta_\Gamma(B)$,
and 
there is a point $x\in \xi\cap 2B$ such that
\begin{align}\label{e:comparing-L-to-beta-dist}
d(x,L_{\tau'})>100\epsilonSubS\beta_\Gamma(B)\diam(\tau')>100\beta(\tau')\diam(\tau')
\end{align}
 then 
 there is  a sub-arc $\check{\xi}\subset \xi$ with image inside $2B$  of diameter
 $$\diam(\check{\xi})>20 \epsilonSubS\beta_\Gamma(B)\diam(B)$$ 
 so that  
$$ d(\check{\xi},\tau')> 20 \epsilonSubS\beta_\Gamma(B)\diam(B)$$ 
\end{lemma}
\begin{proof}
First recall  that $\diam(\xi)\geq r$  and $\xi\cap B\neq \emptyset$. Thus,  
as $\epsilonSubS$ is sufficiently small, 
it suffices to show that
$d(x,\tau')>40 \epsilonSubS\beta_\Gamma(B)\diam(B)$ to get the lemma. We now check this:

Equation \eqref{e:comparing-L-to-beta-dist} together with Lemma \ref{l:lemma-8} yield
\begin{align*}
d(x,\tau')>100\epsilonSubS\beta_\Gamma(B)\diam(\tau') - \beta(\tau')\diam(\tau')=99\epsilonSubS\beta_\Gamma(B)\diam(\tau')
\end{align*}
The lemma then follows as $\diam(\tau')\geq 2r\geq  \diam(B)$.

\end{proof}

\begin{lemma}\label{l:large-sum-1}
Suppose $\xi,\tau,B,Q,r$ are as in Lemma \ref{l:large-subsarc}. 
In addition suppose that 
$\tau\ni\cent(B)$, and that $\tilde{\tau}$ is as in Lemma \ref{l:diam-large-NEW}. 
Let $E$ be the parts of the images of $\xi$ and $\tilde{\tau}$ inside $2B$.
Then the following holds.
If we cover $E$ with balls $\{B_i\}$ such that $\diam(B_i)<10\epsilonSubS\beta_\Gamma(B)\diam(B)$ then 
\begin{align*}
\sum\limits_i \diam(B_i) \geq 4r + \epsilonSubS\beta_\Gamma(B)\diam(B)
\end{align*}
\end{lemma}
\begin{proof}
First note that a ball $B_i$ above can only intersect at most one of the images of $\tilde{\tau}$ or $\check{\xi}$.
We now use the conclusions of Lemmas  \ref{l:diam-large-NEW} and \ref{l:large-subsarc} as follows.

\begin{align*}
\sum\limits_i \diam(B_i) 
&=
\sum\limits_{B_i\cap\tilde{\tau}\neq\emptyset} \ \diam(B_i) + 
\sum\limits_{B_i\cap\check{\xi}\neq\emptyset} \ \diam(B_i) \\
&\geq \diam(\tilde{\tau}) + \diam(\check{\xi})\geq 
4r + (-10  +20)\epsilonSubS\beta_\Gamma(B)\diam(\tau')\\
&\geq
4r + \epsilonSubS\beta_\Gamma(B)\diam(B)
\end{align*}
In the last inequality we  used that $\diam(\tau')\geq 2r \geq \diam(B)$.
\end{proof}

The lemmas above combine together to give the following proposition.
\begin{proposition}\label{p:large-sum-over-kids}
Suppose $B\in \flatballs$ with radius $r$, and $Q=Q(B)$.
$\Lambda(Q)\ni\tau\ni\cent(B)$. 
Then there is a $\xi\in \Lambda(Q)$ such that 
if $\tilde{\tau}$ is  as in   Lemma \ref{l:diam-large-NEW} and
$E$ is  the parts of the images of $\xi \cup \tilde\tau$ inside $2B$ as in Lemma \ref{l:large-subsarc},
then, the following holds.
If 
we cover $E$ with balls $\{B_i\}$ such that 
$\diam(B_i)<10\epsilonSubS\beta_\Gamma(B)\diam(B)$, then
\begin{align}\label{e:1-large-sum-over-kids}
\sum\limits_i \diam(B_i) 
\geq 
4r + \epsilonSubS\beta_\Gamma(B)\diam(B).
\end{align}
\end{proposition}
\begin{proof}
Let $\tau'$ and $\xi'$  denote the respective extensions of $\tau$ and $\xi$ to arcs in $\Lambda'(B)$. 
First, $B\in \flatballs$ implies that  $\beta(\tau')<\epsilonSubS\beta_\Gamma(B)$.
Since $\epsilonSubS\leq 10^{-10}< \frac{\diam(B)}{\diam(\tau')}$
we have
that $\Gamma\cap B$ contains something other than the image of $\tau$, and more specifically, there is an arc $\xi\in \Lambda(Q)$ and a point $x\in B$ which is in the image of $ \xi$ such that 
\begin{align*}
\dist(x,L_{\tau'})\geq \beta_\Gamma(B)\diam(B)
\end{align*}
and since $100\epsilonSubS$ is smaller than the ratio $\diam(B)/\diam(\tau')$ we have, 
\begin{align*}
\dist(x,L_{\tau'})\geq \beta_\Gamma(B)\diam(B)
>
100\epsilonSubS\beta_\Gamma(B)\diam(\tau')
\end{align*}
Thus,  we may apply 
Lemma \ref{l:large-subsarc} and Lemma \ref{l:large-sum-1}, to get the proposition.
\end{proof}

\subsection{A geometric martingale}\label{s:martingale}
Fix an integer $M\geq 0$.
We will set $\cB^M$ to be balls for which we have control over $\beta_\Gamma(B)$ and that we can apply Proposition \ref{p:large-sum-over-kids} to, i.e
$$\cB^M:=\{2B\in \flatballs:  
	\beta_\Gamma(B)\in[2^{-M-1},2^{-M}]\}\,.$$
We also set $J_M$ to be the smallest integer larger than $M-\log(10\epsilonSubS)+10$, and apply Lemma \ref{l:number-of-families} to $\cB^M$ with $J=J_M$ and $\kappa=3$ (the constant $C$ for that lemma will be $2A=20$).
We thus have $\cB^M=\cB^M_1\cup...\cup\cB^M_{D_M}$, where 
$D_M=D(C=2A=20,\kappa=3) \cdot J_M$, which grows linearly in $M$.
Fix $\cB'=\cB^M_i$ for some $i\in \{1,...,D_M\}$ and
apply the construction following Lemma \ref{l:number-of-families}. We call the resulting dyadic-like cubes $\Delta=\Delta(\cB^M,i)$.
We will use the properties of  Lemma \ref{cube-properties-1} below.

Below we denote $\cH^1_\Gamma(F):=\cH^1(F\cap \Gamma)$.
The following proposition is as easy consequence of Proposition \ref{p:large-sum-over-kids} above.

\begin{proposition}\label{prop:large-sum-redone}
Let $2B\in 2\flatballs$ be given.  Suppose $Q=Q(B)\in \Delta$, is written as
\begin{align}\label{e:chop-up}
Q=(\cup_i Q^i)\cup R_Q\,, 
\end{align}
where $Q^i=Q(B^i)\in \Delta$ are maximal such that $Q^i\subsetneq Q$,
and $R_Q$ is chosen so that the union above is disjoint.
Then,
$$\sum_i \diam(Q^i) + \cH^1_\Gamma(R_Q)\geq    \diam(Q)\left(1+  \frac1{10}\epsilonSubS\beta_\Gamma(B)\right)$$
\end{proposition}
\begin{proof}
Let $\alpha=40\epsilonSubS2^{-M-10}$.
Using Lemma \ref{cube-properties-1} we have  that 
$Q\subset 2(1+\alpha)B$ as well as $Q^i\subset 2(1+\alpha)B^i$.
Now, recalling that we also have $\beta_\Gamma(B)\in[2^{-M-1},2^{-M}]$, we have
\begin{align*}
\sum_i \diam(Q^i) + \cH^1_\Gamma(R_Q)
&\geq
\sum_{2(1+\alpha)B^i\cap E\neq \emptyset} \diam(2B^i)  + \cH^1_\Gamma(R_Q) \\
&\overset{\eqref{e:diam-scale}}{\geq}
\frac1{1+\alpha}
\sum_{2(1+\alpha)B^i\cap E\neq \emptyset} \diam(2(1+\alpha)B^i)  + \cH^1_\Gamma(R_Q) \\
&\overset{\eqref{e:1-large-sum-over-kids}} \geq
\frac1{1+\alpha}
(4r + \epsilonSubS\beta_\Gamma(B)\diam(B))\\
&\geq
4r(1-\alpha) + (1-\alpha)\epsilonSubS\beta_\Gamma(B)\diam(B)\\
&\geq 
4r + \frac12 \epsilonSubS\beta_\Gamma(B)\diam(B) \\
&\geq \diam(Q)\left(1 + \frac1{10} \epsilonSubS\beta_\Gamma(B)\right)
\end{align*}
\end{proof}

We can now show the main proposition for this section.

\begin{proposition}\label{p:martingale-prop}
$$\sum\limits_{Q\in\Delta} \diam(Q) \leq \frac{10}{\epsilonSubS} 2^M \cH^1(\Gamma)$$
\end{proposition}
\begin{proof}
In the same manner as \cite{Schul-TSP,Schul-AR} 
we define positive function $w_Q:\H\to \R$ such that
\begin{enumerate}[(i)]
\item
$\int_Q w_Qd\cH^1_\Gamma \geq \diam(Q)$
\item
For almost all $x\in \Gamma$, 
$$\sum\limits_{Q\in \Delta} w_Q(x)\leq \frac{10}{\epsilonSubS} 2^M$$
\item $w_Q$ is supported inside $Q$
\end{enumerate}
The functions $w_Q$ will be constructed as a martingale.
Denote $w_Q(Z)=\int_Z w_Qd\cH^1_\Gamma$.
Set
$$
w_Q(Q)=\diam(Q) .
$$
Assume now that $w_Q(Q')$ is defined.  We define $w_Q(Q'^i)$ and $w_Q(R_{Q'})$,
where
$$ 
Q'=(\cup  Q'^i) \cup R_{Q'},
$$
a decomposition as given by equation \eqref{e:chop-up}.

Take
$$
w_Q(R_{Q'})=\frac{w_Q( Q')}{s'} \cH^1_\Gamma(R_{Q'}) 
$$
(uniformly distributed)
and
$$
w_Q( Q'^i)=\frac{w_Q(Q')}{s'}\diam(Q'^i),
$$
where 
$$
s'=\cH^1_\Gamma(R_{Q'})+\sum_i \diam(Q'^i).
$$
This will give us $w_Q$.
Note that 
$s'\leq 2 \cH^1(\Gamma\cap Q')$.
Clearly (i) and (iii) are satisfied.
Furthermore, 
If $x\in R_{Q'}$, we have from (a rather weak use of) Proposition \ref{prop:large-sum-redone} that
\begin{equation}\label{e:r-Q-estimate}
w_Q(x)\leq \frac{w_Q(Q')}{s'}\leq \frac{w_Q(Q')}{\diam(Q')}\,. 
\end{equation}
To see (ii), note that for any $j$ we may write:
\begin{eqnarray*}
\frac{w_Q( Q'^{j})}{\diam(Q'^{j}) }
&=&
\frac{w_Q( Q')}{s'}\\
&=&
\frac{w_Q( Q')}{\diam(Q') }
\frac{\diam(Q' )}{s'}\\
&=&
\frac{w_Q( Q')}{\diam(Q') }
\frac{\diam(Q') }
	{\cH^1_\Gamma(R_{Q'}) + 
		\sum\limits_{i} \diam(Q'^i) }\\
&\leq&
\frac{w_Q( Q')}{\diam(Q') }
\frac{1}
	{1+ c_02^{-M}}\,,\\
\end{eqnarray*}
where $c_0=\frac1{10}\epsilonSubS$ is obtained from  Proposition \ref{prop:large-sum-redone}.

And so,
\begin{eqnarray*}
\frac{w_Q( Q'^{j})}{\diam(Q'^{j})} \leq 
 	q   \frac{w_Q( Q')}{\diam(Q')}
\end{eqnarray*}
with $q=\frac{1}
{1+ c_02^{-M}}$.  Now, suppose 
that  $x\in Q_N \subset ...\subset Q_1$.
we  get:
\begin{eqnarray*}
\frac{w_{Q_1}(Q_N)}{\diam(Q_N)} &\leq& 
  q\frac{w_{Q_1}(Q_{N-1})}{\diam(Q_{N-1})} \\
  &\leq&...\\
  &\leq&
  q^{N-1}\frac{w_{Q_1}(Q_{1})}{\diam(Q_1)}=q^{N-1}.
\end{eqnarray*}
We have using \eqref{e:r-Q-estimate} that for 
$x\in R_{Q_{N}}$
\begin{equation}
w_{Q_1}(x) \leq   \frac{w_{Q_1}(Q_N)}{\diam(Q_N)} \leq  q^{N-1}. 
\end{equation}
Let $E$ denote the collection of all  elements $x$ which are in an infinite sequence of $\Delta$ i.e. can be written as elements $x\in .... \subset Q_N \subset ...\subset Q_1$,  for any positive integer $N$.
Then,
as $\cH^1_\Gamma(Q)\geq r(B(Q))\geq \frac15\diam(Q)$,
we have that for any $N$
\begin{equation}
w_{Q_1}(Q_N) \leq   \diam(Q_N) q^{N-1} \leq  5q^N\cH^1_\Gamma(Q_N)
\end{equation}
which yields that for $\cH^1_\Gamma$-almost-every $x\in E$ we have that $w_{Q_1}(x)=0$.

This will give us (ii) as a sum of a geometric series  since
$$\sum q^n = \frac1{1-q}
\leq  \frac1{c_0 2^{-M}}=\frac{10}{\epsilonSubS}2^{M}.$$

Now,
\begin{eqnarray*}
\sum\limits_{Q \in \Delta}\diam(Q)
&=&
\sum\limits_{Q \in \Delta}\int w_{Q}(x)d\cH^1_\Gamma(x)\\
&=&
\int \sum\limits_{Q \in \Delta} w_{Q}(x)d\cH^1_\Gamma(x)\\
&\leq&
\frac{10}{\epsilonSubS}
\int  2^{M} d\cH^1_\Gamma(x)
=
\frac{10}{\epsilonSubS}
2^{M}\cH^1(\Gamma).
\end{eqnarray*}

\end{proof}

\begin{proof}[Proof of inequality \eqref{e:G2-control}]

We will  show the stronger inequality
\begin{align*}
\sum\limits_{B\in \flatballs} \beta_\Gamma(B)^2\diam(B)
\leq C\cH^1(\Gamma)\,.
\end{align*}
Recall the discussion at the start of Section \ref{s:martingale}.  There, for an integer $M\geq 0$, we get (using Lemma  \ref{l:number-of-families}) for  
$i\in\{1,...,D_M\}$ a subset $\cB_i^M\subset 2\flatballs$.  We apply the construction  which follows  Lemma  \ref{l:number-of-families} to $\cB_i^M$, and get $\Delta(\cB^M,i)$. 
Then
\begin{eqnarray*}
\sum\limits_{B\in \flatballs} \beta_\Gamma(B)^2\diam(B)
&\leq& 
\sum\limits_{M\geq 0} 
	\sum\limits_{\ 2B\in \cB^M}
		 (2^{-M})^2\diam(B)\\
&\leq& 
\sum\limits_{M\geq 0} 
	\sum\limits_{i=1}^{D_M}
	\sum\limits_{2B\in \cB_i^M }
		 2^{-2M}\diam(B)\\
&\leq& 
\sum\limits_{M\geq 0} 2^{-2M} 
	\sum\limits_{i=1}^{D_M}
	\sum\limits_{Q\in \Delta(\cB^M,i)}
		 \diam(Q)\\
&\leq& 
\sum\limits_{M\geq 0} 2^{-2M} 
	\sum\limits_{i=1}^{D_M}
	\frac{10}{\epsilonSubS}2^M\cH^1(\Gamma)	
\end{eqnarray*}
where for the last inequality, we used  Proposition \ref{p:martingale-prop}.
Thus, we reduce to the calculation 
\begin{align*}
\sum_{M\geq 0} D_M  2^{-2M} 2^M 
\leq 
\sum_{M\geq 0} D J_M 2^{-M} 
\leq
\sum_{M\geq 0} D \cdot (1+M-\log(10\epsilonSubS)+10) 2^{-M} 
<\infty
\end{align*}
where the last finite bound  is independent of $\Gamma$.
\begin{note}
Recall that as per the start of Section \ref{s:martingale}, 
$D$ is a constant that depends on $\kappa=3$ as well as the constant $A=10$ (fixed in Note \ref{n:fix-a}). 
The constant $\epsilonSubS$ is fixed in Note \ref{n:epsilon-s} to be $10^{-10}$.
\end{note}

\end{proof}

\section*{PART B}
\section{Curvature estimates for the Heisenberg group}\label{s:curvature-in-Heisenberg}

The purpose of this section is to prove Proposition \ref{p:prop-4}.  It is independent from the rest of the paper.  The only properties of the Heisenberg group we will use is the exact formula for the Koranyi metric, the invariance of the Koranyi metric under group multiplication, rotation about the $z$-axis, and that the Koranyi metric scales under the dilation automorphisms.  All of these properties hold no matter what $\eta$ is.
We will need the following simple numerical inequality.

\begin{lemma} \label{concave-power}
  Let $p \geq 1$ and $a,b > 0$.  If $b \geq 2^p a$ then
  \begin{align*}
    (a + b)^{1/p} \geq a^{1/p} + \frac{1}{2} b^{1/p}.
  \end{align*}
\end{lemma}

\begin{proof}
  \begin{align*}
    (a+b)^{1/p} \geq \frac{b^{1/p}}{2} + \frac{b^{1/p}}{2} \geq a^{1/p} + \frac{b^{1/p}}{2}.
  \end{align*}
\end{proof}

We also will need a lemma that allows us to reduce finding a lower bound of the triangle inequality to finding the lower bound of a power of the triangle inequality.

\begin{lemma} \label{power-curvature}
  Let $a,b,c \in \H$ so that
  \begin{align*}
    \max\{d(a,b),d(b,c)\} \leq \alpha d(a,c),
  \end{align*}
  for some $\alpha \geq 1/2$.  Then
  \begin{align*}
    d(a,b) + d(b,c) - d(a,c) \geq \frac{1}{100\alpha^3 d(a,c)^3} \left[(d(a,b) + d(b,c))^4 - d(a,c)^4\right].
  \end{align*}
\end{lemma}

\begin{proof}

Let $t=\frac{d(a,b) + d(b,c)}{d(a,c)}$ and $M=\frac{1}{ d(a,c)^4}\left((d(a,b) + d(b,c))^4 - d(a,c)^4\right)$.  The lemma will follow if we show that
if $t\leq 2\alpha$ and
$$t^4-1\geq M$$
then
$$t-1\geq \frac{M}{100\alpha^3}.$$
Indeed, 
$$t-1=\frac{t^4-1}{(t+1)(t^2+1)}\geq \frac{M}{(1+2\alpha)(1+4\alpha^2)} \geq \frac{M}{100\alpha^3}\,.$$
In the last inequality, we used the fact that $\alpha \geq 1/2$.
\end{proof}

We can now prove  Proposition \ref{p:prop-4}.

\begin{proof}[Proof of Proposition \ref{p:prop-4}]
  For convenience, we set $D = \diam(\{p_1,p_2,p_3,p_4\})$.  The proof will consist of many case analyses of the four points depending on their configuration.  We will use decimals to demarcate subcases, so case 2.1.2 is a subcase of 2.1 is a subcase of case 2.

  Before we start the case analyses, we first prove the general fact that
  \begin{align}
    \frac{1}{6^4 \cdot (3D)^3} \max_{i \in \{1,2,3\}} \sup_{a \in \overline{p_ip_{i+1}}} d(a, \overline{p_1p_4})^4 \leq d(p_1,p_2) + d(p_2,p_3) + d(p_3,p_4). \label{segments-dist-upper-bound}
  \end{align}
  Indeed, as $d(p_1,p_2) + d(p_2,p_3) + d(p_3,p_4) \leq 3D$, it further reduces to showing when $i \in \{1,2,3\}$ that
  \begin{align*}
    \sup_{a \in \overline{p_ip_{i+1}}}d(a,\overline{p_1p_4}) \leq 6\left(d(p_1,p_2) + d(p_2,p_3) + d(p_3,p_4)\right).
  \end{align*}
  This is straightforward as for all $t \in [0,1]$ we have
  \begin{multline*}
    d(p_i \delta_t \tilde{\pi}(p_i^{-1}p_{i+1}), \overline{p_1p_4}) \leq d(p_i \delta_t \tilde{\pi}(p_i^{-1}p_{i+1}), p_1) \leq d(p_i \delta_t \tilde{\pi}(p_i^{-1}p_{i+1}), p_i) + d(p_i,p_1) \\
    \leq d(p_i,p_{i+1}) + d(p_i,p_1).
  \end{multline*}
  Here, we've used the fact that $N(\tilde{\pi}(g)) \leq N(g)$ for all $g \in \H$.  We now proceed case by case.

  {\bf Case 1:} $d(p_1,p_2) + d(p_2,p_3) + d(p_3,p_4) > \frac{3}{2} d(p_1,p_4)$.

  We then have that
  \begin{align}
    d(p_1,p_2) + d(p_2,p_3) + d(p_3,p_4) - d(p_1,p_4) \geq \frac{1}{3} (d(p_1,p_2) + d(p_2,p_3) + d(p_3,p_4)). \label{large-broken-line}
  \end{align}
  Equations \eqref{large-broken-line} and \eqref{segments-dist-upper-bound} give \eqref{curvature-ineq} as $\epsilon < 1$ and $\eta < 1$, which finishes the proof of this case.

  {\bf Case 2:} We can now suppose
  \begin{align}
    d(p_1,p_2) + d(p_2,p_3) + d(p_3,p_4) \leq \frac{3}{2} d(p_1,p_4). \label{bounded-excess}
  \end{align}
  {
  Note that the inequality we are trying to prove is invariant with respect to isometries and scales with dilation.  Indeed, the terms in \eqref{curvature-ineq} are all stated in terms of relative distance and both sides are 1-homogeneous with respect to dilation.  One just has to verify that the horizontal line segment interpolants $\overline{p_ip_{i+1}}$ behave well under these operations.  Verifying that they behave well under translation and rotation is trivial (that is, $g\overline{p_ip_{i+1}} = \overline{(gp_i)(gp_{i+1})}$ and $R_\theta \overline{p_ip_{i+1}} = \overline{R_\theta(p_i) R_\theta(p_{i+1})}$).  It is also easy to prove that they scale properly with dilation.  Indeed, for $s \in [0,1]$ and $\lambda > 0$, we have
  \begin{align*}
    \delta_\lambda (g \delta_s \tilde{\pi}(g^{-1}h)) = \delta_\lambda(g) \delta_\lambda \delta_s\tilde{\pi}(g^{-1}h) = \delta_\lambda(g) \delta_s \delta_\lambda \tilde{\pi}(g^{-1}h) = \delta_\lambda(g) \delta_s \tilde{\pi}(\delta_\lambda(g)^{-1} \delta_\lambda(h)).
  \end{align*}
  
  Thus, having proven that \eqref{curvature-ineq} is invariant under isometries and scales with dilation, we are free to normalize $p_1,p_2,p_3,p_4$ using these operations.  We will suppose that $p_1 = (0,0,0)$ by translation.  We may suppose that that $p_1$ and $p_4$ do not project to the same point under $\pi$ as we could have perturbed the points initially by an infinitesimally small amount to put them in general position without affecting the bound by too much.  Thus, we may suppose that $p_4 = (1,0,t)$ by rotation and dilation.  We cannot apply any more operations without changing either $p_1$ or $p_4$ so we will have to write $p_2 = (x,y,z)$, $p_3 = (u,v,w)$.  Note that under this normalization, we have $d(p_1,p_4) = (1+ \eta t^2)^{1/4}$.}

  {\bf Case 2.1:} $d(p_1,p_4) > 100/\epsilon^2$.

  We first state the intuition for this subcase.  Because we have fixed the projection of $p_4$ to $\R^2$ as $(1,0)$, saying that $d(p_1,p_4)$ is large is saying $p_1$ and $p_4$ are very vertical with respect to each other.  Note that the Koranyi metric behaves like the square root metric for such points.  We will seek to obtain the needed excess from the fact that the triangle inequality is very generous for the square root metric when points are spread out.  The case when two points are very close together requires a separate analysis.

  Let $R = d(p_1,p_4)$.  We then have that
  \begin{align}
    \frac{1}{18^4D^3} \max_{i \in \{1,2,3\}} \sup_{a \in \overline{p_ip_{i+1}}} d(a,\overline{p_1p_4})^4 \overset{\eqref{segments-dist-upper-bound} \wedge \eqref{bounded-excess}}{\leq} R. \label{segments-dist-upper-bound-2}
  \end{align}
  
  We have that
  \begin{align}
    |t| = \left( \frac{R^4-1}{\eta} \right)^{1/2} = \left( \frac{R^4 - 1}{R^4} \right)^{1/2} \frac{R^2}{\eta^{1/2}} \geq \left(1 - \frac{\epsilon^2}{1000}\right) \frac{R^2}{\eta^{1/2}}. \label{suppose-large-t}
  \end{align}
  Here, we've used the hypothesis of case 2.1 (in a very non-sharp manner).

  {\bf Case 2.1.1:} $\min \left\{ |z|, \left| t + \frac{y}{2} - z \right| \right\} > \frac{\epsilon^2}{16} \frac{R^2}{\eta^{1/2}}$.  This is the case when $p_2$ is vertically far from both $p_1$ and $p_4$.

  Then
  \begin{align*}
    d(p_1,p_2) + d(p_2,p_3) + d(p_3,p_4) &\geq d(p_1,p_2) + d(p_2,p_4) \\
    &\geq \eta^{1/4} |z|^{1/2} + \eta^{1/4} \left| t + \frac{y}{2} - z \right|^{1/2} \\
    &\geq \eta^{1/4} \left( \left|t + \frac{y}{2} \right| + 2 |z|^{1/2} \left| t + \frac{y}{2} - z \right|^{1/2} \right)^{1/2} \\
    &\geq \eta^{1/4} \left( |t| - \frac{|y|}{2} + \frac{\epsilon^2}{8} \frac{R^2}{\eta^{1/2}} \right)^{1/2} = (*).
  \end{align*}
  Here, we've used the triangle inequality along with the hypothesis of case 2.1.1.  As $((x^2+y^2)^2 + \eta z^2)^{1/4} < 3R/2$ by \eqref{bounded-excess}, we must have that $|y| < 3R/2$.  We then get
  \begin{align*}
    (*) \overset{\eqref{suppose-large-t}}{\geq} \eta^{1/4} \left[ \left(1 - \frac{\epsilon^2}{1000} \right) \frac{R^2}{\eta^{1/2}} - \frac{3\eta^{1/2}}{4R} \frac{R^2}{\eta^{1/2}} + \frac{\epsilon^2}{8} \frac{R^2}{\eta^{1/2}} \right]^{1/2} \geq \left(1 + \frac{\epsilon^2}{16} \right) R.
  \end{align*}
  Here, we used the fact that $\eta < 1$ and $R > 100/\epsilon^2$.  This proves the proposition as the right hand side of \eqref{curvature-ineq} is bounded by a multiple of $R$, as we proved in \eqref{segments-dist-upper-bound-2}.

  {\bf Case 2.1.2:} $\min \left\{ |z|, \left| t + \frac{y}{2} - z \right| \right\} < \frac{\epsilon^2}{16} \frac{R^2}{\eta^{1/2}}$.  This is now the case when $p_2$ is vertically close to one of $p_1$ and $p_4$.

  We first suppose that $|z| < \frac{\epsilon^2}{16} \frac{R^2}{\eta^{1/2}}$, that is $p_2$ is vertically close to $p_1$ and so the horizontal component of $p^{-1}p_2$ must be dominant.  Indeed, as $((x^2 + y^2)^2 + \eta z^2)^{1/4} > \epsilon R$ by \eqref{eps-lower-bound}, we must have that
  \begin{align}
    (x^2 + y^2)^{1/2} \geq \frac{\epsilon}{2} R, \label{large-horizontal-1}
  \end{align}
  and so $(x^2 + y^2)^2 \geq 2^4 \eta z^2$ by our upper bound on $|z|$.  By an application of Lemma \ref{concave-power}, we have
  \begin{align*}
    d(p_1,p_2) + &d(p_2,p_3) + d(p_3,p_4) \\
    &\geq d(p_1,p_2) + d(p_2,p_4) \\
    &= ((x^2 + y^2)^2 + \eta z^2)^{1/4} + \left(((1-x)^2 + y^2)^2 + \eta \left(t + \frac{y}{2} - z\right)^2 \right)^{1/4} \\
    &\geq \frac{1}{2} (x^2 + y^2)^{1/2} + \eta^{1/4} \left( |z|^{1/2} + \left| t + \frac{y}{2} - z \right|^{1/2} \right) = (*).
  \end{align*}
  Remembering that $|y| \leq 3R/2$, we can continue
  \begin{align*}
  (*) \overset{\eqref{large-horizontal-1}}{\geq} \frac{\epsilon}{4} R + \eta^{1/4} \left| t + \frac{y}{2} \right|^{1/2} \overset{\eqref{suppose-large-t}}{\geq} \frac{\epsilon}{4}R + \left( 1 - \frac{\epsilon^2}{1000} - \frac{3\eta^{1/2}}{4R} \right)^{1/2} R \geq \left( 1 + \frac{\epsilon}{8} \right) R.
  \end{align*}
  In the last inequality, we needed to use the fact that $R > 100/\epsilon^2$ and $\eta < 1$.  As before, this proves the proposition as the right hand side of \eqref{curvature-ineq} is bounded by a multiple of $R$, as we proved in \eqref{segments-dist-upper-bound-2}.

  The case when $\left|t + \frac{y}{2} - z \right| < \frac{\epsilon^2}{16} \frac{R^2}{\eta^{1/2}}$ is treated in a similar manner.  This would represent the case when $p_2$ is vertically close to $p_4$.  This finishes the analysis of case 2.1.

  {\bf Case 2.2:} $d(p_1,p_4) \leq 100/\epsilon^2$.

  In particular, we have that
  \begin{align}
    \eta t^2 \leq d(p_1,p_4)^4 \leq \left(\frac{10}{\epsilon}\right)^8. \label{eta-t-bound}
  \end{align}
  Note that as $d(p_1,p_2) \leq \frac{3}{2} d(p_1,p_4)$ and $d(p_3,p_1) \leq d(p_3,p_2) + d(p_2,p_1) \leq \frac{3}{2} d(p_1,p_4)$ by \eqref{bounded-excess}, we get the following bounds:
  \begin{align}
    |x| &\leq \frac{150}{\epsilon^2}, \label{small-x-1} \\
    |u| &\leq \frac{150}{\epsilon^2}. \label{small-u-1}
  \end{align}

  Recall our normalization that $p_1 = (0,0,0)$, $p_4 = (1,0,t)$, $p_2 = (x,y,z)$, and $p_3 = (u,v,w)$.  For $(a,b,c) \in \H$, let $(a,b,c)_x = (a,0,0)$ denote the projection onto the $x$-axis.  The triangle inequality then gives that
  \begin{align*}
   & \sup_{s \in [0,1]} d(p_2 \delta_s \tilde{\pi}(p_2^{-1}p_3), \overline{p_1p_4})^4 \\
   &\leq \left(\sup_{s \in [0,1]} d(p_2 \delta_s \tilde{\pi}(p_2^{-1}p_3), (p_2 \delta_s \tilde{\pi}(p_2^{-1}p_3))_x) + d((p_2 \delta_s \tilde{\pi}(p_2^{-1}p_3))_x,\overline{p_1p_4})\right)^4 \\
   &\leq 8\left(\sup_{s \in [0,1]} d(p_2 \delta_s \tilde{\pi}(p_2^{-1}p_3), (p_2 \delta_s \tilde{\pi}(p_2^{-1}p_3))_x)^4 + d((p_2 \delta_s \tilde{\pi}(p_2^{-1}p_3))_x,\overline{p_1p_4})^4\right) \\
    &\leq \sup_{s \in [0,1]} \left[ 8\left( y + (v-y) s \right)^4 + 8 \eta \left( z - \frac{xy}{2} - (uy + 2xy)s - \frac{1}{2} (uv + xy - uy - xv) s^2 \right)^2 \right] \\
    &\qquad + 8\max\{(x-1)_+,(-x)_+,(u-1)_+,(-u)_+ \}^4.
  \end{align*}
  Here, we have the function $r_+ = \max\{r,0\}$.  Using the fact that $(u+v)^p \leq 2^{p-1} (|u|^p + |v|^p)$, $\eta < 1$, and inequalities \eqref{small-x-1} and \eqref{small-u-1}, we get (by an overestimation) that
  \begin{align}
    \sup_{a \in \overline{p_2p_3}} d(a,\overline{p_1p_4})^4 \leq \frac{10^{10}}{\epsilon^4} \max\{y^4,v^4,z^2,y^2,v^2, ((x-1)_+)^4,((-x)_+)^4,((u-1)_+)^4,((-u)_+)^4\}. \label{listed-bound}
  \end{align}
  In the same way, we also have that
  \begin{align}
    \sup_{a \in \overline{p_1p_2}} d(a,\overline{p_1p_4})^4 &\leq \frac{10^{10}}{\epsilon^4} \max\{ y^4, y^2, ((-x)_+)^4, ((x-1)_+)^4 \}, \label{listed-bound-1} \\
    \sup_{a \in \overline{p_3p_4}} d(a,\overline{p_1p_4})^4 &\leq \frac{10^{10}}{\epsilon^4} \max\{ v^4, v^2, w^2, ((-u)_+)^4, ((u-1)_+)^4 \}. \label{listed-bound-2}
  \end{align}
  
  We now claim that, to prove the proposition under the current case hypotheses, we can reduce to proving that for any $\eta < (\epsilon/10)^{10}$, we get that
  \begin{align}
    (d(p_1,p_2) + d(p_2,p_4))^4 - d(p_1,p_4)^4 &\geq \frac{1}{4} \eta^2 \max\{y^4+y^2+((x-1)_+)^4+((-x)_+)^4,z^2\}, \label{q1-curvature-ineq} \\
    (d(p_1,p_3) + d(p_3,p_4))^4 - d(p_1,p_4)^4 &\geq \frac{1}{4} \eta^2 \max\{v^4+v^2+((u-1)_+)^4+((-u)_+)^4,w^2\}. \label{q2-curvature-ineq}
  \end{align}
  Indeed, by \eqref{bounded-excess} and the triangle inequality, we get that
  \begin{align*}
    \max\{d(p_1,p_2),d(p_2,p_4)\} \leq d(p_1,p_2) + d(p_2,p_4) \leq \frac{3}{2} d(p_1,p_4), \\
    \max\{d(p_1,p_3),d(p_3,p_4)\} \leq d(p_1,p_3) + d(p_3,p_4) \leq \frac{3}{2} d(p_1,p_4).
  \end{align*}
  As $d(p_1,p_4) \leq D$, by an application of Lemma \ref{power-curvature} with $\alpha = \frac{3}{2}$, we get that proving \eqref{q1-curvature-ineq} and \eqref{q2-curvature-ineq} would give (after overestimation)
  \begin{align}
    d(p_1,p_2) + d(p_2,p_4) - d(p_1,p_4) &\geq \frac{\eta^2}{10000D^3}\max\{y^4,y^2,((x-1)_+)^4,((-x)_+)^4,z^2\}, \label{q1-curvature-ineq'} \\
    d(p_1,p_3) + d(p_3,p_4) - d(p_1,p_4) &\geq \frac{\eta^2}{10000D^3}\max\{v^4,v^2,((u-1)_+)^4,((-u)_+)^4,w^2\}. \label{q2-curvature-ineq'}
  \end{align}
  Here, we've also used the fact that $\max\{a_1,...,a_n\} \leq a_1 + ... + a_n$ for nonnegative $a_i$.  A simple application of the triangle inequality gives
  \begin{multline*}
    d(p_1,p_2) + d(p_2,p_3) + d(p_3,p_4) - d(p_1,p_4) \geq \max_{i \in \{2,3\}} (d(p_1,p_i) + d(p_i,p_4) - d(p_1,p_4)) \\
    \overset{\eqref{q1-curvature-ineq'} \wedge \eqref{q2-curvature-ineq'}}{\geq} \frac{\eta^2}{10000D^3} \max\{y^4,v^4,z^2,y^2,v^2,z^2,w^2, ((x-1)_+)^4,((-x)_+)^4,((u-1)_+)^4,((-u)_+)^4\}.
  \end{multline*}
  Appealing to \eqref{listed-bound}, \eqref{listed-bound-1}, and \eqref{listed-bound-2} now proves the proposition.

  {Note that the inequalities \eqref{q1-curvature-ineq} and \eqref{q2-curvature-ineq} should not be viewed as ``general inequalities'' as the terms on the right hand side are reflecting our normalization of $p_1,p_2,p_3,p_4$.  }
  
  Thus, it suffices to prove \eqref{q1-curvature-ineq} and \eqref{q2-curvature-ineq}.  
  We will only prove \eqref{q1-curvature-ineq}, which comes in two steps: one lower bounding the left hand side by $\frac{1}{2} \eta^2(y^4+y^2+((x-1)_+)^4+((-x)_+)^4)$ and one lower bounding by $\frac{1}{4}\eta^2 z^2$.  The proof of \eqref{q2-curvature-ineq} follows the exact same structure with only $p_3$ replacing $p_2$.

  Before we start the proof, let us describe the intuition behind the proof.  As before, there will be many case analyses (although some cases will resemble others).  Our first case to rule out is when the $y$-component of $p_2$ is large (Cases 2.2.1A and 2.2.2B.1).  As $p_4 = (1,0,t)$, this would mean that the three point configuration, $\{p_1,p_2,p_4\}$ is highly unaffine when projected onto $\R^2$.  Then, assuming $|y|$ is large enough, the normal Euclidean curvature inequality would give the needed lower bounds.  Thus, we may assume that $p_2$ lies close to the $xz$-plane.  We now use the reasoning behind case 2.1.  If $p_2$ is vertically far from $p_1$ and $p_4$, then we hope to gain our lower bound from the excess of the triangle inequality in the square root metric.  These two cases are in Cases 2.2.2A.1 and 2.2.2B.2.1 and will be handled in a similar manner that Case 2.1.1 was handled.  Otherwise, $p_2$ is vertically close to one of the points $p_1$ or $p_4$, say $p_1$, and so \eqref{eps-lower-bound} says that the horizontal component of $p_1^{-1}p_2$ must be large.  We then use Lemma \ref{concave-power} to derive our lower bound.
  
  We remind the reader of the reverse Minkowski inequality, which we will use many times to group the inequalities by components:
  \begin{align*}
    \left( \sum_i a_i^{1/q} \right)^q + \left( \sum_i b_i^{1/q} \right)^q \leq \left(\sum_i (a_i + b_i)^{1/q}\right)^q.
  \end{align*}
  This inequality holds whenever $a_i$ and $b_i$ are nonnegative numbers and $q \geq 1$.

  We will use A and B to denote the subcases is the two lower bounds that we need.  Note that A and B are not meant to be seen as mutually exclusive.  So 2.2.1A is disjoint from 2.2.2A, but has no relation to 2.2.1B.
  
  
{\bf A}:  $\mathbf{\frac{1}{4} \eta^2(y^4 + y^2 + ((x-1)_+)^4 + ((-x)_+)^4)}$ {\bf lower bound}.
   By expanding the $(x^2+y^2)^2$ and $((1-x)^2+y^2)^2$ terms and using the reverse Minkowski's inequality, we have
  \begin{align}
    &(d(p_1,p_2)+d(p_2,p_4))^4 \\
    &\qquad = \left[\left((x^2+y^2)^2 + \eta z^2\right)^{1/4} + \left(((1-x)^2+y^2)^2 + \eta\left( t + \frac{y}{2} - z \right)^2\right)^{1/4} \right]^4 \notag \\
    &\qquad = \left[\left(x^4 + x^2y^2 +y^4 + (x^2y^2 + \eta z^2)\right)^{1/4} \right. \notag \\
    &\qquad\qquad + \left.\left((1-x)^4+(1-x)^2y^2 + y^4 + \left((1-x)^2y^2+ \eta\left( t + \frac{y}{2} - z \right)^2 \right) \right)^{1/4} \right]^4 \notag \\
    &\qquad\geq \left(|x| + |1-x|\right)^4 + \left(|xy|^{1/2} + |(1-x)y|^{1/2}\right)^4 + 16 y^4 \notag \\
    &\qquad\qquad  + \left[(x^2y^2 + \eta z^2)^{1/4} + \left((1-x)^2y^2 + \eta \left( t + \frac{y}{2} - z \right)^2 \right)^{1/4} \right]^4. \label{reverse-minkowski-2}
  \end{align}
  We can easily calculate
  \begin{align}
    (|x| + |1-x|)^4 = (1 + 2(-x)_+ + 2(x-1)_+)^4 \geq 1 + 16((-x)_+)^4 + 16((x-1)_+)^4. \label{x-excess}
  \end{align}
  Note that
  \begin{align}
    \left( |xy|^{1/2} + |(1-x)y|^{1/2} \right)^4 \geq x^2y^2 + (1-x)^2y^2 \geq  \frac{1}{2} y^2. \label{always-y}
  \end{align}
  Indeed, this follows from the fact that $x^2 + (1-x)^2 \geq \frac{1}{2}$ always.  We therefore get
  \begin{multline*}
    (d(p_q,p_2)+d(p_2,p_4))^4 \overset{\eqref{reverse-minkowski-2} \wedge \eqref{x-excess} \wedge \eqref{always-y}}{\geq} 1 + 16((-x)_+)^4 + 16((x-1)_+)^4 + \frac{1}{2} y^2 + y^4 \\
    + \left[(x^2y^2 + \eta z^2)^{1/4} + \left((1-x)^2y^2 + \eta \left( t + \frac{y}{2} - z \right)^2 \right)^{1/4} \right]^4,
  \end{multline*}
  and as $\eta < 1$, it then suffices to prove that
  \begin{align*}
    \left[(x^2y^2 + \eta z^2)^{1/4} + \left((1-x)^2y^2 + \eta \left( t + \frac{y}{2} - z \right)^2 \right)^{1/4} \right]^4 \geq \eta t^2.
  \end{align*}
  By another application of the reverse Minkowski's inequality, we have
  \begin{align}
    &\left[(x^2y^2 + \eta z^2)^{1/4} + \left((1-x)^2y^2 + \eta \left( t + \frac{y}{2} - z \right)^2 \right)^{1/4} \right]^4 \notag \\
    &\qquad \geq \left(|xy|^{1/2} + |(1-x)y|^{1/2}\right)^4 + \eta \left( |z|^{1/2} + \left| t + \frac{y}{2} - z \right|^{1/2} \right)^4 \label{second-minkowski} \\
    &\qquad \geq y^2 + \eta \left( t + \frac{y}{2} \right)^2. \notag
  \end{align}
  
  {\bf Case 2.2.1A:} $|y| > \eta |t|$.  This is the case when the projection of $\{p_1,p_2,p_4\}$ is highly unaffine.
  
  We have that
  \begin{align*}
    y^2 + \eta\left( t + \frac{y}{2} \right)^2 \geq \eta t^2 + y^2 + \eta ty \geq \eta t^2.
  \end{align*}
  This gives the lower bound needed and finishes this case.

  {\bf Case 2.2.2A:} We can now suppose that
  \begin{align}
    |y| \leq \eta |t|. \label{small-y-1}
  \end{align}
  Then we also have
  \begin{align}
    \left(|z|^{1/2} + \left| t + \frac{y}{2} - z \right|^{1/2}\right)^2 &= |z| + \left| t + \frac{y}{2} - z\right| + 2|z|^{1/2} \left| t + \frac{y}{2} -z\right|^{1/2} \notag \\
    &\geq |t| - \frac{|y|}{2} + 2|z|^{1/2} \left| t + \frac{y}{2} - z\right|^{1/2}. \label{square-out-sqrt}
  \end{align}

  {\bf Case 2.2.2A.1:} $z \notin \left[ -\frac{y^2}{4|t|}, \frac{y^2}{4|t|} \right] \cup \left[ t + \frac{y}{2} - \frac{y^2}{4|t|}, t + \frac{y}{2} + \frac{y^2}{4|t|} \right]$.  This is the case when $p_2$ is vertically far from both $p_1$ and $p_4$.

  Suppose first that $|z| \leq \frac{1}{2} \left| t + \frac{y}{2} \right|$.  Then
  \begin{align*}
    |z|^{1/2} \left| t + \frac{y}{2} - z \right|^{1/2} \geq \left( \frac{y^2}{4|t|} \right)^{1/2} \left| \frac{t}{2} + \frac{y}{4} \right|^{1/2} \overset{\eqref{small-y-1}}{\geq} \left( \frac{y^2}{4|t|} \right)^{1/2} \left| \frac{t}{4} \right|^{1/2} \geq \frac{|y|}{4}.
  \end{align*}
  In the penultimate inequality, we used the fact that $\eta < 1$ to get that $|y| < |t|$ from \eqref{small-y-1}.  This together with \eqref{second-minkowski} and \eqref{square-out-sqrt} gives our needed lower bound.
  
  For the case when $|z| \geq \frac{1}{2} \left| t + \frac{y}{2} \right|$, the same proof works with the roles of $|z|^{1/2}$ and $\left| t + \frac{y}{2} - z \right|^{1/2}$ reversed.

  {\bf Case 2.2.2A.2:} $z \in \left[ -\frac{y^2}{4|t|}, \frac{y^2}{4|t|} \right] \cup \left[ t + \frac{y}{2} - \frac{y^2}{4|t|}, t + \frac{y}{2} + \frac{y^2}{4|t|} \right]$.

  Suppose first that $|z| \leq \frac{y^2}{4|t|}$, that is, $p_2$ is vertically close to $p_1$.  We then have
  \begin{align}
    |y| &\overset{\eqref{small-y-1}}{\leq} \eta |t| = \left( \eta^2 t^2 \right)^{1/2} \overset{\eqref{eta-t-bound}}{\leq} \eta^{1/2} \left( \frac{10}{\epsilon} \right)^4, \label{small-y-2} \\
    |z| &\leq \frac{y^2}{4|t|} \overset{\eqref{small-y-1}}{\leq} \frac{1}{4} \eta^2 |t| \overset{\eqref{eta-t-bound}}{\leq} \frac{1}{4} \eta^{3/2} \left( \frac{10}{\epsilon} \right)^4. \label{small-z}
  \end{align}
  Thus, since we have chosen $\eta < (\epsilon/10)^{10}$ and remembering that $d(p_1,p_4)^4 = 1 + \eta t^2 \geq 1$, we get that
  \begin{align*}
    |x| = \left( \left( d(p_1,p_2)^4 - \eta z^2 \right)^{1/2} - y^2 \right)^{1/2} \overset{\eqref{eps-lower-bound}}{\geq} \left( \left( \epsilon^4 - \eta z^2 \right)^{1/2} - y^2 \right)^{1/2} \overset{\eqref{small-y-2} \wedge \eqref{small-z}}{\geq} \frac{\epsilon}{2}.
  \end{align*}
  As $|z| \leq \frac{y^2}{4|t|} \leq \frac{1}{4} \eta |y|$, we have by our choice of $\eta$ that
  \begin{align*}
    2^4 \eta z^2 \leq 2^4 \frac{1}{16} \eta^3 y^2 \leq \frac{1}{4} \epsilon^2 y^2 \leq x^2 y^2.
  \end{align*}
  Then we can use Lemma \ref{concave-power} to show that
  \begin{align}
    (x^2y^2 + \eta z^2)^{1/4} \geq \frac{1}{8} \epsilon^{1/2} |y|^{1/2} + \eta^{1/4} |z|^{1/2}. \label{large-xy-z-gap}
  \end{align}
  Now we have that
  \begin{align*}
    &\left[(x^2y^2 + \eta z^2)^{1/4} + \left((1-x)^2y^2 + \eta \left( t + \frac{y}{2} - z \right)^2 \right)^{1/4} \right]^4 \\
    &\qquad \overset{\eqref{large-xy-z-gap}}{\geq} \left[\frac{1}{8} \epsilon^{1/2} |y|^{1/2} + \eta^{1/4} |z|^{1/2} + \eta^{1/4} \left| t + \frac{y}{2} - z \right|^{1/2} \right]^4 \\
    &\qquad \geq \left[\frac{1}{4} \epsilon^{1/2} |y|^{1/2} + \eta^{1/4} \left| t + \frac{y}{2} \right|^{1/2} \right]^4 \\
    &\qquad > \left[\eta^{1/4} |y|^{1/2} + \eta^{1/4} \left| t + \frac{y}{2} \right|^{1/2} \right]^4 \\
    &\qquad \geq \eta t^2.
  \end{align*}
  In the penultimate inequality, we used the fact that we have chosen $\eta < (\epsilon/10)^{10}$.

  Thus, we may suppose $\left| t + \frac{y}{2} - z \right| \leq\frac{y^2}{4|t|}$.  We can then simply repeat the argument with $|z|$ in place of $\left| t+ \frac{y}{2} - z\right|$ and $(1-x)$ in place of $x$.  This is the case when $p_2$ is vertically close to $p_4$.  The only problem will be to show that $1-x$ is sufficently large.  To do this, we will use the fact that we have supposed
  \begin{align}
    |y| &\overset{\eqref{small-y-1}}{\leq} \eta |t| \overset{\eqref{eta-t-bound}}{\leq} \eta^{1/2} \left(\frac{10}{\epsilon}\right)^4, \label{e:abs-y-bnd} \\
    \left| t + \frac{y}{2} - z \right| &\leq \frac{y^2}{4|t|} \overset{\eqref{eta-t-bound} \wedge \eqref{small-y-1}}{\leq} \frac{1}{4} \eta^{3/2}\left( \frac{10}{\epsilon} \right)^4. \label{e:t+y/2-bnd}
  \end{align}
  in conjunction with
  \begin{align}
    \left((1-x)^2 + y^2\right)^2 + \eta \left( t + \frac{y}{2} - z \right)^2 \overset{\eqref{eps-lower-bound}}{\geq} \epsilon^4(1 + \eta t^2) \geq \epsilon^4 \label{e:large-1-x}
  \end{align}
  and $\eta < (\epsilon/10)^{10}$ to get that
  \begin{multline*}
    |1-x| \overset{\eqref{e:abs-y-bnd} \wedge \eqref{e:t+y/2-bnd} \wedge \eqref{e:large-1-x}}{\geq} \left[ \left( \epsilon^4 - \frac{1}{16} \eta^4 \left( \frac{10}{\epsilon} \right)^8 \right)^{1/2} - \eta \left( \frac{10}{\epsilon} \right)^8 \right]^{1/2} \\
    \geq \left[ \left( \epsilon^4 - \frac{1}{16} \left( \frac{\epsilon}{10} \right)^{32} \right)^{1/2} - \frac{\epsilon^2}{100} \right]^{1/2} \geq \frac{\epsilon}{2}.
  \end{multline*}
  This allows us to continue as was done previously.  This finishes case 2.2.2A, which finishes the lower bound associated with $y^4 + y^2 + ((x-1)_+)^4 + ((-x)_+)^4$.

  {\bf B:} $\mathbf{\frac{1}{4} \eta^2 z^2}$ {\bf lower bound}.
  
  {\bf Case 2.2.1B:} $t^2 \leq \frac{1}{2} z^2$.
  
  As $d(p_1,p_2) \geq |x|$ and $d(p_2,p_4) \geq |1-x|$, we then have
  \begin{multline*}
    (d(p_1,p_2) +d(p_2,p_4))^4 \geq d(p_1,p_2)^4 + 4|x|^3 |1-x| + 6 |x|^2 |1-x|^2 + 4 |x| |1-x|^3 + |1-x|^4 \\
    \geq 1 + \eta z^2 \geq 1 + \eta t^2 + \frac{\eta}{2} z^2 > 1 + \eta t^2 + \frac{1}{4} \eta^2 z^2.
  \end{multline*}
  In the last inequality, we used the fact that $\eta < 1$.

  {\bf Case 2.2.2B:} $t^2 > \frac{1}{2} z^2$.

  We will prove instead that when $\eta < (\epsilon/10)^{10}$, we have that
  \begin{align}
    (d(p_1,p_2) + d(p_2,p_4))^4 - d(p_1,p_4)^4 \geq \frac{1}{2} \eta^2 t^2. \label{t2-lower-bound}
  \end{align}
  By the hypothesis of the current subcase, this clearly suffices.

  As in the calculations that led up to \eqref{reverse-minkowski-2}, an application of the reverse Minkowski's inequality gives us
  \begin{align}
    (d(p_1,p_2) + &d(p_2,p_4))^4 \notag \\
    &\geq \left( |x|+ |1-x| \right)^4 + 2\left( |xy|^{1/2} + |(1-x)y|^{1/2} \right)^4 + \eta \left[ |z|^{1/2} + \left| t + \frac{y}{2} - z \right|^{1/2} \right]^4 \notag \\
    &\geq 1 + 2\left( |xy|^{1/2} + |(1-x)y|^{1/2} \right)^4 + \eta \left[ |z|^{1/2} + \left| t + \frac{y}{2} - z \right|^{1/2} \right]^4 \label{fourth-minkowski} \\
    &\geq 1 + 2y^2 + \eta \left( t + \frac{y}{2} \right)^2. \notag
  \end{align}

  {\bf Case 2.2.2B.1:} $|y| > \eta |t|$.  This is the case when the projection of $\{p_1,p_2,p_4\}$ to $\R^2$ is highly unaffine.

  An easy calculation gives
  \begin{align*}
    2y^2 + \eta \left( t + \frac{y}{2}\right)^2 \geq \eta t^2 + 2y^2 + \eta t y > \eta t^2 + 2 \eta |t y| + \eta t y > (\eta + \eta^2) t^2,
  \end{align*}
  which proves the needed inequality.
  
  {\bf Case 2.2.2B.2:} We may now suppose
  \begin{align}
    |y| \leq \eta |t|. \label{small-y-3}
  \end{align}
  As before, we have that
  \begin{align}
    \left( |z|^{1/2} + \left| t + \frac{y}{2} - z \right|^{1/2} \right)^2 &\geq |z| + \left| t + \frac{y}{2} - z \right| + 2 |z|^{1/2} \left| t + \frac{y}{2} - z \right|^{1/2} \notag \\
    &\overset{\eqref{small-y-3}}{\geq} \left(1 - \frac{\eta}{2} \right)|t| + 2 |z|^{1/2} \left| t + \frac{y}{2} - z \right|^{1/2} \label{square-out-sqrt-2} \\
  \end{align}

  {\bf Case 2.2.2B.2.1:} $z \notin \left[ -\eta^2|t|, \eta^2|t| \right] \cup \left[ t + \frac{y}{2} - \eta^2|t|, t + \frac{y}{2} + \eta^2|t| \right]$.  This is the case when $p_2$ is vertically far from $p_1$ and $p_4$.

  Suppose first that $|z| \leq \frac{1}{2} \left| t + \frac{y}{2} \right|$.  Then
  \begin{align*}
    |z|^{1/2} \left| t + \frac{y}{2} - z \right|^{1/2} \geq \eta |t|^{1/2} \left| \frac{t}{2} + \frac{y}{4} \right|^{1/2} \overset{\eqref{small-y-3}}{\geq} \eta |t|^{1/2} \left| \frac{t}{4} \right|^{1/2} \geq \frac{\eta}{2} |t|.
  \end{align*}
  As before, we used the fact that $\eta < 1$ to get that $|y| < |t|$ from \eqref{small-y-3}.  This together with \eqref{fourth-minkowski} and \eqref{square-out-sqrt-2} gives our needed lower bound.

  For the case when $|z| \geq \frac{1}{2} \left| t + \frac{y}{2} \right|$, the same proof works with the roles of $|z|^{1/2}$ and $\left| t + \frac{y}{2} - z \right|^{1/2}$ reversed.  This completes the lower bound in this subcase.

  {\bf Case 2.2.2B.2.2:} $z \in \left[ -\eta^2|t|, \eta^2|t| \right] \cup \left[ t + \frac{y}{2} - \eta^2|t|, t + \frac{y}{2} + \eta^2|t| \right]$.

  We will first suppose that $p_2$ is vertically close to $p_1$:
  \begin{align}
    |z| \leq \eta^2|t|. \label{small-z-1}
  \end{align}
  Then, as before, we have
  \begin{align*}
    |y| &\overset{\eqref{small-y-3}}{\leq} \eta |t| \overset{\eqref{eta-t-bound}}{\leq} \eta^{1/2} \left( \frac{10}{\epsilon} \right)^4,\\
    |z| &\leq \eta^2 |t| \overset{\eqref{eta-t-bound}}{\leq} \eta^{3/2} \left( \frac{10}{\epsilon} \right)^4.
  \end{align*}
  As before, because we have taken $\eta < (\epsilon/10)^{10}$, we get that
  \begin{align*}
    |x| \geq \frac{\epsilon}{2}.
  \end{align*}
  Remembering that $d(p_1,p_2) \geq |x|$ and $d(p_2,p_4) \geq \left(|1-x|^4 + \eta \left( t + \frac{y}{2} - z \right)^2 \right)^{1/4} \geq |1-x|$, we have that
  \begin{align*}
    (d(p_1,p_2) + &d(p_2,p_4))^4 \\
    &\geq x^4 + 4x^3(1-x) + 6x^2\left((1-x)^4 + \eta \left( t + \frac{y}{2} - z \right)^2 \right)^{1/2} + 4x(1-x)^3 + (1 - x)^4 \\
    &\qquad + \eta \left( t + \frac{y}{2} - z \right)^2\\
    &\overset{\eqref{small-y-3} \wedge \eqref{small-z-1}}{\geq} 1 + 6x^2\left((1-x)^4 + \eta \left( 1 - \frac{\eta}{2} - \eta^2 \right)^2 t^2 \right)^{1/2} - 6x^2(1-x)^2 \\
    &\qquad + \eta \left( 1 - \frac{\eta}{2} - \eta^2 \right)^2 t^2 \\
    &\geq 1 + 6x^2\left((1-x)^4 + (\eta - 2\eta^2 ) t^2 \right)^{1/2} - 6x^2(1-x)^2 + \eta t^2 - 2 \eta^2 t^2.
  \end{align*}
  In the last inequality, we had to use the fact that $\eta < 1$.  Thus, to prove \eqref{t2-lower-bound} it suffices to show that
  \begin{align}
    6x^2\left((1-x)^4 + (\eta - 2\eta^2 ) t^2 \right)^{1/2} \geq 6x^2(1-x)^2 + \frac{5}{2} \eta^2 t^2 \label{final-t-ineq-step}
  \end{align}

  {\bf Case 2.2.2B.2.2.1:} $4 (1 - x)^4 \geq (\eta - 2\eta^2)t^2$.
  
  First note that 
  when $4a \geq b>0$, by  concavity of square root we have that  
    $$(a+b)^{1/2} = a^{1/2} \left( 1 + \frac{b}{a} \right)^{1/2} \geq a^{1/2} \left( 1 + \frac{(4+1)^{1/2}-1}{4} \frac{b}{a} \right) \geq a^{1/2} + \frac{b}{4a^{1/2}}.$$
  Then, using the hypothesis of this subcase, we get that
  \begin{align*}
    6x^2\left((1-x)^4 + (\eta - 2\eta^2 ) t^2 \right)^{1/2} &\geq 6x^2 (1-x)^2 + \frac{6}{4} \frac{x^2}{(1-x)^2} (\eta - 2\eta^2)t^2 \\
    &\geq 6x^2 (1-x)^2 + \frac{\epsilon^2}{100} (\eta - 2\eta^2)t^2.
  \end{align*}
  In the last inequality, we used the fact that $|x| \geq \frac{\epsilon}{2}$.
  As we have chosen $\eta < (\epsilon/10)^{10}$, we get that
  \begin{align*}
    \frac{\epsilon^2}{100} (\eta - 2\eta^2)t^2 \geq \frac{5}{2} \eta^2 t^2,
  \end{align*}
  proving \eqref{final-t-ineq-step}.
  
  {\bf Case 2.2.2B.2.2.2:} $4(1-x)^4 < (\eta - 2\eta^2)t^2$.
  
  Then by Lemma \ref{concave-power} and the fact that $|x| > \frac{\epsilon}{2}$, we have that
  \begin{align*}
    6x^2\left((1-x)^4 + (\eta - 2\eta^2 ) t^2 \right)^{1/2} &\geq  6x^2 (1-x)^2 + \frac{3}{4} \epsilon^2 (\eta - 2\eta^2)^{1/2} |t|. 
  \end{align*}
  Thus, to prove \eqref{final-t-ineq-step}, it suffices to show
  \begin{align*}
    \frac{3}{4} \epsilon^2 (\eta - 2\eta^2)^{1/2} |t| \geq \frac{5}{2} \eta^2 t^2, \qquad \forall t \overset{\eqref{eta-t-bound}}{\in} \left[-\frac{1}{\eta^{1/2}} \left( \frac{10}{\epsilon} \right)^4, \frac{1}{\eta^{1/2}} \left( \frac{10}{\epsilon} \right)^4 \right].
  \end{align*}
  Put another way, we are being asked to show (after using the bound $(\eta - 2\eta^2)^{1/2} \geq \frac{1}{2} \eta^{1/2}$) that
  \begin{align*}
    |t| \leq \frac{1}{\eta^{1/2}} \left( \frac{10}{\epsilon} \right)^4 \Longrightarrow |t| \leq \frac{3}{20} \frac{\epsilon^2}{\eta^{3/2}}.
  \end{align*}
  This follows because we have chosen $\eta < (\epsilon/10)^{10}$.  Thus, we have proven the $t^2$ bound when $|z| \leq \eta^2|t|$.

  For the case when $\left| t + \frac{y}{2} - z \right| \leq \eta^2|t|$ ({\it i.e.} $p_2$ is vertically close to $p_4$), we can proceed as in the $y^4 + y^2 + ((x-1)_+)^4 + ((-x)_+)^4$ case by just repeating the above steps with $|z|$ in place of $\left| t + \frac{y}{2} - z\right|$ and $1-x$ in place of $x$.  We can use the same argument as in the $y^4 + y^2 + ((x-1)_+)^4 + ((-x)_+)^4$ case to show that $|1-x|$ must be sufficiently large.
  
  This now finishes the proof for the $t^2$ bound, which finishes the $z^2$ bound, which also finishes the proof of Case 2.2 and thus the entire proposition.

\end{proof}

\end{document}